\documentclass[11pt,reqno,a4paper]{amsart}

\usepackage[margin=1in]{geometry}
\usepackage{graphicx}
\usepackage{amsaddr}
\usepackage{graphicx, float, changepage, subcaption} 
\usepackage[hidelinks]{hyperref} 
\usepackage{color}
\usepackage{amssymb, amsmath, dsfont,amsthm}
\usepackage{cleveref}
\usepackage{enumerate}
\usepackage{tabls}

\newcommand{\eps}{\varepsilon}

\renewcommand{\d}{\, \mathrm{d}}
\newcommand{\integral}{\mathrm{int}}
\newcommand{\EM}{\mathrm{EM}}
\newcommand{\rmE}{\mathrm{E}}
\newcommand{\rmL}{\mathrm{L}}
\newcommand{\E}{ \mathbb{E}}
\newcommand{\G}{ \mathcal{G}}
\newcommand{\grd}{\nabla}

\newcommand{\GmaN}{\mathcal{G}^{\mathrm{E}}_N}
\newcommand{\GmiN}{\mathcal{G}^{\mathrm{L}}_N}
\newcommand{\dist}{\operatorname{dist}}

\newtheorem{theorem}{Theorem}
\newtheorem{example}{Example}
\newtheorem{remark}{Remark}
\newtheorem{assumption}{Assumption}
\newtheorem{proposition}{Proposition}
\newtheorem{corollary}{Corollary}
\newtheorem{lemma}{Lemma}
\makeatletter
\def\namedlabel#1#2{\begingroup
	\def\@currentlabel{#2}%
	\phantomsection\label{#1}\endgroup
}
\makeatother
\makeatletter
\renewcommand{\paragraph}{%
	\@startsection{paragraph}{4}{\z@}%
	{1.5ex plus 0.5ex minus 0.2ex}%
	{-1em}%
	{\normalfont\normalsize\bfseries}%
}
\makeatother

\title[Two Adjoint Perspectives on Fokker-Planck: Micro-Macro Correspondence]{Two Adjoint Perspectives on Fokker-Planck Optimization: A Microscopic-Macroscopic Correspondence}

\author{Kathrin Hellmuth$^{a}$, Qin Li$^{b}$, Yunan Yang$^{c}$
}
\address{$^{a}$Department of Computing and Mathematical Sciences, California Institute of Technology, USA\\
	$^{b}$ Department of Mathematics, University of Wisconsin-Madison, USA\\
	$^{c}$  Department of Mathematics, Cornell University, USA}

\email{hellmuth@caltech.edu}
\email{qinli@math.wisc.edu}
\email{yunan.yang@cornell.edu}

\begin{document}

	\maketitle
	
	\begin{abstract}
	The Fokker--Planck equation admits both a macroscopic Eulerian description through probability densities and a microscopic Lagrangian description through stochastic trajectories. Consequently, optimization problems constrained by the Fokker--Planck equation can be formulated from either perspective. Surprisingly, the corresponding adjoint equations appear to be fundamentally different: the macroscopic adjoint is governed by the backward Kolmogorov equation, whereas the microscopic adjoint evolves pathwise along stochastic trajectories. In this note, we reconcile these two formulations by establishing their correspondence in the continuum setting. We further show that, although their discrete gradients no longer coincide after discretization, both provide consistent numerical approximations of the continuum gradient. Explicit convergence rates are established for both discretization strategies.
	\end{abstract}
	\begin{quote}
		\noindent 
		{\small {\bf Keywords:} Fokker--Planck equation, adjoint methods, stochastic particle methods, Eulerian vs. Lagrangian perspective, optimize-then-discretize, PDE-constraint optimization, Monte Carlo approximation }
	\end{quote}
	
	\section{Introduction}

	The Fokker--Planck equation is one of the most fundamental evolution equations describing the transport of probability distributions under drift and diffusion. Its adjoint equation plays a central role in inverse problems, stochastic optimal control, mean-field games, and PDE-constrained optimization. For the PDE-constrained optimization problem
	\[
	\min_a \mathcal{J}(\rho_T),
	\]
	with $\rho$ satisfying the prototype equation
	\begin{equation}\label{eq:FP}
		\partial_t\rho_t+\nabla\cdot(a(x)\rho_t)=\Delta\rho_t,
	\end{equation}
	on $x\in \mathbb R^d, t\in (0,T]$, the associated adjoint is the backward Kolmogorov equation
	\begin{equation}\label{eq:adjoint}
		\partial_t\mu_t+a(x)\cdot\nabla\mu_t+\Delta\mu_t=0\,,\quad\text{with}\quad \mu_T= \frac{\delta\mathcal{J}}{\delta\rho_T}({\rho_T})\,,
	\end{equation}
	and the gradient with respect to the control variable is assembled through the forward and adjoint variables as
	\begin{equation}\label{eq:gradient}
		\G[a](x)=\frac{\delta\mathcal J}{\delta a}(x)
		=
		\int_0^T
		\rho(t,x)\nabla\mu(t,x)\d t.
	\end{equation}
	This gradient term is then deployed in an optimization iteration. The main object examined in this paper is this gradient term at any fixed $a$, and throughout the paper, we drop $a$ dependence. When no ambiguity arises, we use shorthand notation $\rho_t$ and $\mu_t$ to denote $\rho(t,\cdot)$ and $\mu(t,\cdot)$.
	
	To numerically compute this gradient, one needs to examine the effects of discretization. We focus on particle-type discretization of the forward equation in sense of Monte Carlo approximation, and study how they affect the computation of~\eqref{eq:gradient}.
	
	More specifically, we adopt the Lagrangian formulation and simulate the Fokker--Planck equation's solution by particles/samples that evolve forward in time according to a stochastic differential equation (SDE),
	\begin{equation}\label{eq:particle}
		\rho_t
		\approx
		\rho_t^N
		=
		\frac1N\sum_{i=1}^N\delta_{X_{i,t}},
		\qquad
		dX_{i,t}
		=
		a(X_{i,t})\d t+\sqrt2\,dW_{i,t}\,.
	\end{equation}
	The empirical measure $\rho_t^N$ converges to the solution of~\eqref{eq:FP}, both in the weak and the Wasserstein sense. If one differentiates the particle formulation~\eqref{eq:particle} directly, a straightforward sensitivity analysis shows that the corresponding particle adjoint satisfies the pathwise ordinary differential equation (ODE)
	\begin{equation}\label{eq:adjoint_particle}
		dY_{i,t}
		=
		-\nabla a(X_{i,t})^\top Y_{i,t}\d t\,,\quad\text{with}\quad Y_{i,T}=\nabla\frac{\delta\mathcal J}{\delta\rho_T}(\rho^N_T)(X_{i,T}).
	\end{equation}
	Surprisingly, this microscopic adjoint bears little resemblance to the backward Kolmogorov equation~\eqref{eq:adjoint}. While~\eqref{eq:adjoint} is a deterministic backward parabolic PDE, ~\eqref{eq:adjoint_particle} is a pathwise ODE evolving along stochastic trajectories. At first sight, it is not even clear why the two formulations should be connected to the same gradient.
	
	This mismatch is reminiscent of the classical distinction between \emph{optimize-then-discretize} (OTD) and \emph{discretize-then-optimize} (DTO). The particle representation can be considered a discretization and the adjoint~\eqref{eq:adjoint_particle} arises from differentiating these discretized dynamics, whereas the Kolmogorov adjoint~\eqref{eq:adjoint} is obtained by first deriving the optimality system at the continuum level, before discretization. The two procedures do not commute, leading to two distinct numerical representations of the continuum gradient~\eqref{eq:gradient}.
	
	This is the central question of the paper:
	
	\medskip
	\begin{center}
		\emph{How can the microscopic and macroscopic adjoint formulations be reconciled?}
	\end{center}
	\medskip
	
	The goal of this note is to clarify this apparent discrepancy. We show that the two formulations represent two complementary differentiation paradigms for the same optimization problem: an Eulerian (PDE-based) perspective and a Lagrangian (particle-based) perspective.
	
	We first remain at the continuum level and introduce the stochastic representation of the Fokker--Planck equation. Instead of viewing the state variable solely as a probability density, we represent it by the stochastic process
	\begin{equation}\label{eq:process}
		dX_t=a(X_t)\d t+\sqrt{2}\,dW_t,     \qquad \mathrm{Law}(X_t)=\rho_t.
	\end{equation}
	Associated with each sample path is the pathwise adjoint variable
	\begin{equation}\label{eq:process_adjoint}
		dY_t=- \nabla a(X_t)^\top Y_t\d t,\quad\text{with}\quad Y_T=\nabla\frac{\delta\mathcal J}{\delta\rho_T}(\rho_T)(X_T),
	\end{equation}
	where $\nabla a(x)\in \mathbb R^{d\times d}$ denotes the Jacobian matrix of $a:\mathbb R^d\to \mathbb R^d$.

	The main results of this paper are threefold.
	
	\begin{itemize}
		\item \textbf{Continuum correspondence.}
		We first show that the continuum gradient admits two equivalent representations,
		\[
		\G(x)=\underbrace{\int_0^T\nabla\mu_t(x) \rho_t(x)\d t}_{\text{Eulerian}}=\underbrace{\mathbb E\left[\int_0^T Y_t\delta_{x-X_t}\d t\right]}_{\text{Lagrangian}}.
		\]
		The key connection between the two formulations is the conditional expectation identity
		\begin{equation}
			\mathbb{E}\!\left(Y_t \mid X_t=x\right)
			=
			\nabla \mu(t,x),
			\qquad \rho_t\text{-a.e.}
		\end{equation}
		which explains why both adjoint formulations produce the same first-order variation. This correspondence is established in Section~2.
		
		\item \textbf{Eulerian discretization.}
		Discretizing the macroscopic formulation yields the numerical gradient
		\[
		\GmaN(x)=\int_0^T\rho^N(t,x)\nabla\mu^N(t,x)\d t,
		\]
		where $\rho^N$ is a numerical approximation of $\rho$, and $\mu^N$ solves the adjoint equation~\eqref{eq:adjoint} with terminal condition determined by $\rho^N(T)$. Under suitable assumptions we prove that for a suitable distance function $\dist(\cdot, \cdot)$, one has
		\[
		\dist(\GmaN,\G)=\mathcal O\!\left(
		W_1(\rho^N,\rho)\right).
		\]
		When $\rho^N$ is generated by the particle approximation~\eqref{eq:particle}, this further gives the Monte Carlo sampling rate from \cite{fournier2015rate}
		\begin{equation}\label{eq:alphaN}
			\alpha(N):=
			\begin{cases}
				N^{-1/2},&d=1,\\
				N^{-1/2}\log(1+N),&d=2,\\
				N^{-1/d},&d>2,
			\end{cases}
		\end{equation}
		and thus
		\[
		\dist(\GmaN,\G)=\mathcal O(\alpha(N))\,.
		\]
		We present the precise topology and assumptions in Section~3.
		
		\item \textbf{Lagrangian discretization.}
		Alternatively, one may discretize the stochastic representation directly by replacing the expectation in the Lagrangian formulation with finitely many sample paths~\eqref{eq:particle}-\eqref{eq:adjoint_particle}. This leads to an intuitive approximation:
		\[
		\GmiN=\frac1N\sum_{i=1}^N\int_0^TY_{i,t}\,\delta_{x-X_{i,t}}
		\d t\,.
		\]
		We will establish
		\[
		\dist(\GmiN,\G)=\mathcal O(\alpha(N))
		\]
		with the precise notion of convergence presented in Section~4.
	\end{itemize}

	It is worth emphasizing that the two numerical approximations are fundamentally different objects. The Eulerian approximation $\GmaN$ is obtained by solving a discretized adjoint PDE and therefore remains a Eulerian quantity throughout. In contrast, the Lagrangian approximation $\GmiN$ is defined directly on stochastic trajectories and never invokes the adjoint PDE. Consequently, the continuum identity~\eqref{eq:adjoint_connection} no longer survives discretization. Indeed, the naive discrete counterpart is false:
	\[
	\nabla\mu^N(t,x)\neq \left.\frac1N\sum_{i=1}^N Y_{i,t}\,\right\rvert\,X_{i,t}=x\, \qquad \rho_t\text{-a.s.}.
	\]
	The OTD and DTO procedures therefore produce genuinely different finite-dimensional adjoints, even though both converge to the same continuum gradient.
	
	In the subsequent sections, we discuss continuum correspondence, followed by  the Eulerian and Lagrangian discretizations. Some a-priori estimates are repeatedly used throughout the error analysis, and we summarize them in Appendix \Cref{sec:gman_helper}.

	\paragraph{Relation to existing work.} The Fokker--Planck equation is one of the most fundamental PDE models and has  been widely studied. The gradient \eqref{eq:gradient} of an observable $\mathcal J(\rho_T)$ appears in a variety of contexts. It's Eulerian form \eqref{eq:gradient} appeared early in statistical mechanics, where  the celebrated fluctuation-dissipation theorem relates the linear response of an equilibrium system to a small perturbation to  temporal correlation functions of the unperturbed system \cite{kubo1957statistical,agarwal1972fluctuation,prost2009generalized}, with applications, e.g., in parameter inference in molecular dynamics \cite{ercole2017accurate}, climate response \cite{leith1975climate,giorgini2025predicting} and active matter \cite{burkholder2019fluctuation}.
	Related optimization formulations are also widely used, especially in the context of optimal control of probability laws, where backward Kolmogorov or Pontryagin-type adjoint are prominent concepts~\cite{AnnunziatoBorzi2018,roy2018fokker,BreitenbachBorzi2020}.
	For deterministic continuity equations, equivalence among Lagrangian, Eulerian, and Kantorovich formulations, as well as their  particle-to-continuum convergence, have also been established~\cite{CavagnariLisiniOrrieriSavare2022}. Our focus is different: rather than comparing optimal controls, we compare two numerical representations of the reduced gradient at a fixed control. 
	
	From the viewpoint of stochastic processes, several distinct objects are commonly referred to as adjoints. Classical stochastic-flow and semigroup-sensitivity formulas represent derivatives of backward Kolmogorov solutions in terms of the Jacobian of the stochastic flow \cite{Kunita1990,ElworthyLi1994}. Numerically, the Feynman--Kac formula has been deployed~\cite{YANG2021110564} as the pathwise representation of backward Kolmogorov. In stochastic optimal control, the adapted adjoint is instead described by a backward stochastic differential equation, whose martingale term enforces adaptedness \cite{PardouxPeng1990}. By contrast, reverse-mode or pathwise differentiation of a realized SDE trajectory produces a random backward ODE such as Equation~\eqref{eq:process_adjoint} below, computed with the Brownian path held fixed \cite{LiWongChenDuvenaud2020,KidgerFosterLiLyons2021, LeburuNurbekyanRuthotto2026}. Because its terminal value contains future information, this pathwise adjoint is generally not adapted. Its conditional projection onto the present state is the Markovian adjoint, which is precisely the mechanism expressed by \eqref{eq:adjoint_connection} in~\Cref{sec:cont}.

	This conditional-projection viewpoint has recently appeared in the literature on adjoint matching for stochastic optimal control. Domingo-Enrich et al.\ introduced a backward ``lean adjoint'' ODE along simulated trajectories and incorporated it into a regression-based stochastic-control method \cite{DomingoEnrichEtAl2025}. Domingo-Enrich and Han subsequently related this pathwise ODE to the adapted adjoint BSDE arising from the stochastic maximum principle and established equality after conditioning on the current state \cite{DomingoEnrichHan2026}. In the additive-noise, terminal-cost setting considered here, their lean adjoint specializes to Equation~\eqref{eq:process_adjoint}, and the corresponding conditional relation specializes to Equation~\eqref{eq:adjoint_connection}. The emphasis of the present work is instead on identifying the PDE gradient with an adjoint-weighted measure and on comparing the two finite-particle approximations resulting from the Eulerian and Lagrangian constructions. 
	
	At the numerical level, the issue is closely connected to the classical optimize--then--discretize (OTD) versus discretize--then--optimize distinction (DTO). For Monte Carlo PDE solvers, the numerical state and adjoint may both be empirical measures, so a formal product of independently sampled forward and adjoint approximations is not defined. In radiative-transfer and kinetic settings, this obstruction has motivated both DTO particle adjoints and correlated OTD estimators that reuse the forward samples and interpret the adjoint as a test function \cite{LiWangYang2023,CaflischYang2024,chen2026inferenceinteractingkernelmeanfield}. One practical DTO strategy is to record the random number generator state~\cite{L_vbak_2024}. Related work on SDE-constrained calibration likewise contrasts law-level Fokker--Planck formulations with direct trajectory-wise Monte Carlo differentiation \cite{BartschDenkVolkwein2024}. The present analysis isolates this distinction for a diffusion process and quantifies how the two finite-dimensional adjoints, although different for fixed $N$, approximate the same continuum gradient.

	\section{Continuum Correspondence}\label{sec:cont}
	
	Before introducing any numerical discretization, we first compare the Eulerian and Lagrangian formulations at the continuum level. The Eulerian description is expressed in terms of the forward Fokker--Planck equation and its backward Kolmogorov adjoint, whereas the Lagrangian description follows the underlying stochastic trajectories together with their pathwise adjoints. The main result of this section shows that these two apparently different constructions yield the same continuum gradient. In particular, we establish the conditional-expectation identity that connects the pathwise adjoint to the spatial gradient of the Kolmogorov adjoint and thereby reconciles the two formulations.
	
	Since the well-posedness of the forward and adjoint problems already requires certain regularity, we summarize our assumptions on $a$, $\rho_0$ and $\mathcal J$. More details are in~\Cref{ssec:apriori:Continuum,ssec:apriori:particle}.

	\begin{assumption}\label{assumptions}
		We impose the following regularity assumptions on the velocity field $a$, the initial data $\rho_0$ and the objective function $\mathcal J$: Let $\theta \in (0,1)$ be a H\"older exponent  such that
		\begin{enumerate}
			\item[(A1)] \makeatletter\protected@edef\@currentlabel{(A1)}\makeatother
			\label{assumptions:a}
			$a\in C^{1,\theta}_b(\mathbb R^d; \mathbb R^d)$,  and 
			\item[(A2)] \makeatletter\protected@edef\@currentlabel{(A2)}\makeatother \label{assumptions:rho0}
			$\rho_0 \in \mathcal P_q(\mathbb R^d)$ is a probability measure over $\mathbb R^d$ with finite $q$-th moment, for $q=2$ when $d>2$ and $q>2$ otherwise.
			\item[(A3$^\mathrm{E}$)] \makeatletter\protected@edef\@currentlabel{(A3$^\mathrm{E}$)}\makeatother\label{assumptions:Eulerian}  
			There exists a polynomial $\mathfrak{p}>0$  such that the first variation $\frac{\delta\mathcal J}{\delta \rho}$ and its derivatives are polynomially bounded and Lipschitz, in the sense that there exists a constant $C_J$ such that for all measures
			$\nu,\nu' \in \mathcal P_q(\mathbb R^d)$
			$$\left.\frac {\delta \mathcal J}{\delta \rho}\right\rvert_{\nu} \in C^{2,\theta}_\mathfrak{p} \quad 
			\text{and} \quad \left\|\left.\frac {\delta \mathcal J}{\delta \rho}\right\rvert_{\nu}- \left.\frac {\delta \mathcal J}{\delta \rho}\right\rvert_{\nu'}\right\|_{C^{2,\theta}_\mathfrak{p}} \leq C_J W_1(\nu,\nu').$$
			Here $C^{k,\theta}_\mathfrak{p}$ denotes the polynomially weighted H\"older space \cite{lorenzi2000optimal}:
			$$C^{k,\theta}_\mathfrak{p}(\mathbb R^d)=\left\{ f:\mathbb R^d \to \mathbb R \mid \frac {D^\beta f} {\mathfrak{p}} \in C^{\theta} \quad \text{ for all multi indices} \beta \text{ with }\left\lvert\beta\right\rvert\leq k\right\}.$$
			\item[(A3$^\mathrm{L}$)] \makeatletter\protected@edef\@currentlabel{(A3$^\mathrm{L}$)}\makeatother\label{assumptions:Lagrangian}
			There exists a constant $C_J$ such that for all  measures $\nu,\nu' \in \mathcal P_q(\mathbb R^d)$
			$$ \grd \left.\frac {\delta \mathcal J}{\delta \rho}\right\rvert_{\nu} \in L^2_\nu \quad 
			\text{and} \quad \left\|\left.\grd \frac {\delta \mathcal J}{\delta \rho}\right\rvert_{\nu}- \left.\grd\frac {\delta \mathcal J}{\delta \rho}\right\rvert_{\nu'}\right\|_{\infty} \leq C_J W_1(\nu,\nu').$$
			\item[(A4)]\makeatletter\protected@edef\@currentlabel{(A4)}\makeatother
			\label{assumptions:Euler:additional} For $\mathfrak p$ as in \ref{assumptions:Eulerian}, the final data lie in
			\begin{equation*}
				\left.\frac {\delta \mathcal J}{\delta \rho}\right\rvert_{\nu} \in C^{3,\theta}_\mathfrak{p} \qquad \text{for any} \qquad \nu \in \mathcal P_q(\mathbb R^d).
			\end{equation*}
		\end{enumerate}
	\end{assumption}

	\begin{remark}\label{rmk:assumptionsEulerLagrange}
		While~\ref{assumptions:a}-\ref{assumptions:Lagrangian} are imposed to ensure that the first variation is well defined, \ref{assumptions:Euler:additional} is imposed for the analysis of the discretization error in later sections. Comparing~\ref{assumptions:Eulerian} and~\ref{assumptions:Lagrangian}, we see that the regularity requirement in the Eulerian and Lagrangian settings are different. In the Lagrangian setting, the problem is formulated as ODEs and the assumptions are much simpler. In the Eulerian case, our assumptions place us in a Schauder fixed-point setting~\cite{lorenzi2000optimal}. Classical choices for the polynomials are $\mathfrak p(x) = C(1+x^{2m})$ for some $m\in \mathbb N$, and the H\"older space includes functions that have polynomial growth: $\left\lvert D^\beta f(x)\right\rvert \leq C \mathfrak{p}(x)$. Note further that if the polynomial degree $2m$ of $\mathfrak{p}$ does not exceed $\frac q 2$, then~\ref{assumptions:Eulerian} implies~\ref{assumptions:Lagrangian}.
	\end{remark}

	\begin{example} \label{rem:assumptions:linear+MMD}
		Assumptions~\ref{assumptions:Eulerian} and \ref{assumptions:Lagrangian} are stated abstractly, but they can be easily examined and are satisfied by a broad class of objective functionals in applications. We list two representative examples.
		\begin{enumerate}
			\item \textit{(linear objective)} For a fixed $\psi:\mathbb R^d\to \mathbb R$, consider
			
			$$\mathcal J(\mu)=\int_{\mathbb R^d} \psi(x)\,\mu(dx) \qquad \Rightarrow \qquad \frac{\delta J}{\delta \mu}(x)=\psi(x),$$
			which makes the stability estimates trivial. Hence, Assumption~\ref{assumptions:Eulerian} reduces to   $\psi \in C^2_{\mathfrak{p}}(\mathbb R^d)$ and Assumption~\ref{assumptions:Lagrangian} holds iff  $\grd \psi \in L^2_\nu$ for every $\nu\in \mathcal P_q$, which trivially holds under a sufficiently tight polynomial growth bound; see \Cref{rmk:assumptionsEulerLagrange}.
			This class covers linear observables such as moments, or pixel readings.
			\item \textit{(smooth kernel energy)}  Let $K:\mathbb R^d \times \mathbb R^d \to \mathbb R$ be $C^1$ and symmetric  and define the MMD energy
			\[
			\mathcal J(\mu)=\frac12\iint_{\mathbb R^d\times\mathbb R^d}
			K(x,y)\,\mu(dx)\mu(dy),
			\]
			One has 
			$$ \left.\frac {\delta \mathcal J}{\delta \rho}\right\rvert_{\nu}=\int_{\mathbb R^d} K(x,y)\,\nu(dy).$$
			If for some $\mathfrak p$ whose degree does not exceed $\frac q 2$, one has
			$$\|\grd K(\cdot,y) - \grd K(\cdot,y')\|_{\infty}\leq L \left\lvert y-y'\right\rvert \qquad \text{and}\qquad \left\lvert \grd K(x,0)\right\rvert \leq \mathfrak p(x),$$
			then Assumption \ref{assumptions:Lagrangian} holds: square integrability follows from \Cref{rmk:assumptionsEulerLagrange} and 
			$$\left\lvert\left.\grd \frac {\delta \mathcal J}{\delta \rho}\right\rvert_{\nu}(x) - \left.\grd \frac {\delta \mathcal J}{\delta \rho}\right\rvert_{\nu'}(x)\right\rvert\leq L\int \left\lvert y-y'\right\rvert  \d \pi(y,y')$$
			for any coupling $\pi$ of $\nu,\nu'$, so in particular for the optimal one.
			
			For Assumption \ref{assumptions:Eulerian}, it is sufficient to have $$\| K(\cdot,y) -  K(\cdot,y')\|_{C^{2,\theta}_\mathfrak{p}}\leq L \left\lvert y-y'\right\rvert  \qquad \text{ and }\qquad  K(\cdot, 0)\in C^{2,\theta}_{\mathfrak p}. $$ 
			
			A special case is the smoothed $L^2$ objective corresponding to a smooth mollifier $\eta_\varepsilon$ 
			\[
			\mathcal J_\varepsilon(\mu)
			=
			\frac12
			\|\eta_\varepsilon*\mu\|_{L^2(\mathbb R^d)}^2,
			\]
			for which
			$
			K(x,y) = \int \eta_\eps(z-x) \eta_\eps(z-y)\d z$. When  the mollifier is constructed from a smooth  function $\phi$ with bounded derivatives via  $\eta_\eps(x) = \eps^{-d}\phi\left(\frac x \eps \right)$, the above assumptions both hold with $\mathfrak p = 1$.
			
		\end{enumerate}
	\end{example}

	We now state our result.
	\begin{theorem}\label{thm:cont_correspondence}
		Let $\rho$ and $\mu$ solve~\eqref{eq:FP}-\eqref{eq:adjoint}, and let $(X_t,Y_t)$ solve~\eqref{eq:process}-\eqref{eq:process_adjoint}, under assumptions~\ref{assumptions:a},\ref{assumptions:rho0}, \ref{assumptions:Eulerian}, \ref{assumptions:Lagrangian}. Then
		\begin{equation}\label{eq:adjoint_connection}
			\mathbb E(Y_t\,\rvert \,X_t=x)
			=
			\nabla\mu(t,x) \qquad \text{$\rho_t$-almost everywhere (a.e.)}.
		\end{equation}
		Moreover, if the continuum gradient is defined by
		\[
		\mathcal G(x)
		=
		\frac{\delta\mathcal J}{\delta a}(x),
		\]
		then it admits two equivalent representations,
		\begin{equation}\label{eq:Gradient_two_perspectives}
			\mathcal G(x)
			=
			\int_0^T
			\nabla\mu_t(x) \rho_t(x)\d t
			=
			\mathbb E
			\left[
			\int_0^T
			Y_t \, \delta_{x-X_t}\d t
			\right].
		\end{equation}
	\end{theorem}

	\paragraph{Discussion.}
	At first sight, Theorem~\ref{thm:cont_correspondence} appears to suggest that the Eulerian and Lagrangian formulations are completely interchangeable, as reflected in the two equivalent gradient representations in~\eqref{eq:Gradient_two_perspectives}. It is worth emphasizing, however, that although the evolution of the pathwise adjoint $Y_t$ depends only on the individual sample path, its terminal condition is prescribed by the entire law $\rho_T=\mathrm{Law}(X_T)$. So it ``sees" information from other paths, and cannot be computed independently without  knowledge of full distribution $\rho_T$. Consequently, the stochastic Lagrangian formulation should not be viewed as a particle approximation, but rather as an alternative representation of the continuum optimization problem. This distinction is precisely what makes the continuum correspondence in Theorem~\ref{thm:cont_correspondence} possible. This relation breaks down after discretization. While $Y_{i,t}$ satisfies the same equation, its final condition is prescribed by $\rho_T^N$, an empirical measure, rather than the entire law $\rho_T$. Hence, the equivalence is no longer preserved.

	\begin{proof}[Proof of Theorem~\ref{thm:cont_correspondence}]
		We first derive the two representations of the continuum gradient. For the Eulerian formulation, perturb the drift by
		\[
		a\longrightarrow a+\delta\tilde a.
		\]
		The linearized Fokker--Planck equation becomes
		\[
		\partial_t\delta\tilde\rho_t
		+
		\nabla\cdot(a\delta\tilde\rho_t)
		+
		\nabla\cdot(\delta\tilde a\,\rho_t)
		=
		\Delta\delta\tilde\rho_t,
		\qquad
		\delta\tilde\rho_0=0.
		\]
		Multiplying this equation by $\mu$, multiplying~\eqref{eq:adjoint} by $\delta\tilde\rho$, integrating by parts, and retaining only the first-order variation yields
		\[
		\delta\mathcal J
		=
		\int
		\delta\tilde a(x)
		\left(
		\int_0^T
		\nabla\mu_t(x) \rho_t(x)\d t
		\right)
		dx,
		\]
		which gives the first representation in~\eqref{eq:Gradient_two_perspectives}.
		
		For the Lagrangian formulation, differentiating the stochastic dynamics gives
		\begin{equation}\label{eq:process_perturbation}
			d\,\delta\tilde X_t
			=
			\nabla a(X_t)\delta\tilde X_t\d t
			+
			\delta\tilde a(X_t)\d t,
			\qquad
			\delta\tilde X_0=0.
		\end{equation}
		The first variation of the terminal objective can be written as
		\[
		\delta\mathcal J=\left\langle
		\frac{\delta\mathcal J}{\delta\rho}(\rho_T),
		\delta\tilde\rho_T
		\right\rangle \,.
		\]
		Since $\rho_T=\mathrm{Law}(X_T)$, it admits the stochastic representation $\rho_T=\mathbb E[\delta_{X_T}]$. Under the perturbation $X_T\mapsto X_T+\delta\tilde X_T$, we therefore have $\delta\tilde\rho_T=\mathbb E\left[\delta_{X_T+\delta\tilde X_T}-\delta_{X_T}\right]$. Consequently, to the leading order in $\delta\tilde X_T$,
		\[
		\begin{aligned}
			\delta\mathcal J=\mathbb E
			\left[
			\left\langle
			\frac{\delta\mathcal J}{\delta\rho}(\rho_T),
			\delta_{X_T+\delta\tilde X_T}
			-
			\delta_{X_T}
			\right\rangle
			\right]\approx
			\mathbb E
			\left[
			\nabla
			\frac{\delta\mathcal J}{\delta\rho}(\rho_T)(X_T)
			\cdot
			\delta\tilde X_T
			\right].
		\end{aligned}
		\]
		Realizing the terminal condition for $Y_T$,
		\[
		\delta\mathcal J
		=
		\mathbb E
		\left[
		Y_T\cdot\delta\tilde X_T
		\right].
		\]
		Combining~\eqref{eq:process_adjoint} and~\eqref{eq:process_perturbation} yields
		\[
		Y_T\cdot\delta\tilde X_T
		=
		\int_0^T
		Y_t\cdot\delta\tilde a(X_t)\d t.
		\]
		Therefore,
		\[
		\begin{aligned}
			\delta\mathcal J
			&=
			\mathbb E
			\left[
			\int_0^T
			Y_t\cdot\delta\tilde a(X_t)\d t
			\right] \\
			&=
			\int
			\delta\tilde a(x)\cdot
			\mathbb E
			\left[
			\int_0^T
			Y_t\delta_{x-X_t}\d t
			\right]
			dx.
		\end{aligned}
		\]
		This proves the Lagrangian representation in~\eqref{eq:Gradient_two_perspectives}.
		
		It remains to establish~\eqref{eq:adjoint_connection}. 
		Let
		\[
		J_{t,s}
		:=
		\frac{\partial X_s^{t,x}}{\partial x}
		\]
		denote the Jacobian of the stochastic flow $(X^{t,x}_s)_{s\geq t}$ associated with the dynamics \eqref{eq:process} and initialized at $X^{t,x}_t=x$. Then
		\[
		\frac{d}{ds}J_{t,s}
		=
		\nabla a(X_s)J_{t,s},
		\qquad
		J_{t,t}=I.
		\]
		The key observation is that the forward Jacobian and the pathwise adjoint preserve their natural dual pairing:
		\[
		\frac{d}{ds}
		\left(
		Y_sJ_{t,s}
		\right)
		=
		0.
		\]
		Consequently,
		\[
		Y_t
		=
		Y_TJ_{t,T}
		=
		\nabla
		\frac{\delta\mathcal J}{\delta\rho}(\rho_T)(X_T)
		\,
		J_{t,T}.
		\]
		On the other hand, the Feynman--Kac representation \cite[Thm 5.7.6]{karatzas1998methods} gives 
		\[
		\mu(t,x)
		=
		\mathbb E
		\left[
		\left.
		\frac{\delta\mathcal J}{\delta\rho}(\rho_T)(X_T)
		\,\right\rvert\,
		X_t=x
		\right].
		\]
		Differentiating with respect to $x$ and applying the chain rule, see \cite{Kunita1990}, gives
		\[
		\nabla\mu(t,x)
		=
		\mathbb E
		\left[
		\left.
		\nabla
		\frac{\delta\mathcal J}{\delta\rho}(\rho_T)(X_T)
		J_{t,T}
		\,\right\rvert\,
		X_t=x
		\right]
		=
		\mathbb E(Y_t\,\rvert \,X_t=x),\qquad \text{$\rho_t$-a.e.}
		\]
		which completes the proof.
	\end{proof}
	
	\section{Eulerian Formulation}\label{sec:Eulerian}
	
	Having established the continuum correspondence between the Eulerian and Lagrangian formulations, we now turn to their numerical realizations. In this section, we consider the optimize-then-discretize (OTD) approach: we first derive the optimality system at the continuum level and then discretize the forward Fokker--Planck equation through a particle approximation. This yields the Eulerian approximation of the gradient 
	\begin{equation}\label{eq:Eulerian_gradient}
		\GmaN(x) := \int_0^T \nabla \mu_t^N(x)\,\rho_t^N(x) \d t, 
	\end{equation}
	where $\rho_t^N$ is the particle approximation of the solution $\rho$ of the Fokker--Planck equation~\eqref{eq:FP} by the empirical measure of the particles $X^i$ from \eqref{eq:particle}. The adjoint $\mu^N$ approximates the continuum $\mu$ in the sense that both satisfy the backward Kolmogorov equation~\eqref{eq:adjoint}, but their terminal conditions are determined by different terminal measures,
	\begin{equation}\label{eq:adjoint:finalData}
		\mu(T,\cdot)
		=
		\left.
		\frac{\delta\mathcal J}{\delta\rho}
		\right\rvert_{\rho_T},
		\qquad\text{and} \qquad 
		\mu^N(T,\cdot)
		=
		\left.
		\frac{\delta\mathcal J}{\delta\rho}
		\right\rvert_{\rho^N_{T}}.
	\end{equation}
	Our objective is to quantify the error
	\[
	e^{\mathrm E}_N=\G-\GmaN.
	\]
	Since $\rho^N$ is an empirical measure, the quantity $\GmaN$ is only a vector-valued measure rather than an absolutely continuous function.
	Consequently, convergence cannot be expected in a strong topology and must instead be understood in the weak sense. Throughout this section, we therefore evaluate the gradient error by testing against $\phi\in C_c^1(\mathbb R^d;\mathbb R^d)$.

	In Section~\ref{subsec:gman_main}, we state  the main convergence theorem. The proof relies on viewing the continuum gradient as a nonlinear operator acting on the forward trajectory. The result can be extended to the case of a fully discrete forward computation, where both Riemann sum integral error and Euler--Maruyama discretization error are taken into account. This extension is collected in Section~\ref{subsec:gman_EM}.

	\subsection{Main Theorem}\label{subsec:gman_main}

	\begin{theorem}\label{thm:GmaN}
		Under Assumption~\ref{assumptions:a}--\ref{assumptions:Eulerian}, let  $\rho^N_t$ be the empirical particle distribution  of the particles in \eqref{eq:particle}. Then the Eulerian approximation of the gradient defined in~\eqref{eq:Eulerian_gradient} converges weakly to the continuum gradient given in~\eqref{eq:Gradient_two_perspectives}. More precisely, for every $\phi\in C_c^1(\mathbb R^d;\mathbb R^d)$, there exists a constant $C>0$ independent of $N$, such that the gradient error is of the same order as the sampling error \eqref{eq:alphaN}: 
		\begin{equation}\label{eq:convergence_GmaN}
			\mathbb E
			\left\lvert 
			\langle
			\mathcal G-\mathcal G_N^{\mathrm E},
			\phi
			\rangle_x
			\right\rvert 
			=\mathbb E
			\left\lvert 
			\langle
			e_N^{\mathrm E},
			\phi
			\rangle_x
			\right\rvert \leq C\alpha(N).
		\end{equation}
	\end{theorem}
	
	The proof of Theorem~\ref{thm:GmaN} relies on a simple but important observation. The continuum gradient naturally defines a nonlinear operator 
	\[
	\Gamma:
	C([0,T];\mathcal P_q(\mathbb R^d))
	\longrightarrow
	\mathcal M(\mathbb R^d),
	\]
	given by
	\[
	\Gamma[\rho]
	=
	\int_0^T
	\nabla\mu_t \rho_t\d t,
	\]
	where $\mu$ solves the adjoint equation~\eqref{eq:adjoint} with terminal condition $\mu_T=\left. \frac{\delta\mathcal J}{\delta\rho}\right\rvert_{\rho_T}$ and we therefore write $\mu = \mu[\rho]$ when there is no ambiguity. Since the adjoint is uniquely determined by the forward trajectory, the gradient depends only on $\rho$, and we may write
	\[
	\mathcal G=\Gamma[\rho],
	\qquad
	\mathcal G_N^{\mathrm E}
	=
	\Gamma[\rho^N].
	\]
	The convergence theorem therefore reduces to establishing the continuity of the operator $\Gamma$ with respect to perturbations of the forward trajectory.
	Then Theorem~\ref{thm:GmaN} is a corollary of the following proposition.
	
	\begin{proposition}[Stability of the Eulerian gradient operator]
		\label{prop:grad:continuity}
		Under Assumption~\ref{assumptions:a}--\ref{assumptions:Eulerian}, the gradient operator 
		\[
		\Gamma:
		C([0,T];\mathcal P_q(\mathbb R^d))
		\longrightarrow
		\mathcal M(\mathbb R^d)
		\]
		is weakly continuous. More precisely, for every
		$\rho,\rho'\in C([0,T];\mathcal P_q(\mathbb R^d))$
		and every test function $\phi\in C_c^1(\mathbb R^d;\mathbb R^d)$,
		\[
		\left\lvert 
		\left\langle
		\Gamma[\rho]-\Gamma[\rho'],
		\phi
		\right\rangle_x
		\right\rvert 
		\le
		C
		\left(
		\int_0^T
		W_1(\rho_t,\rho_t')\d t
		+
		W_1(\rho_T,\rho_T')
		\right),
		\]
		where the constant $C$ depends only on
		$\|\phi\|_{C^1}$,
		$d$,
		$T$,
		the bounds on $a$,
		and the Lipschitz constant of
		$\frac{\delta\mathcal J}{\delta\rho}$.
	\end{proposition}

	Error rates in Theorem~\ref{thm:GmaN} for the particle approximation then follow automatically from the Wasserstein approximation to the underlying dynamics in \Cref{thm:samplingRho}; see, for example, \cite {fournier2015rate}.
	
	\begin{proof}[Proof of \Cref{prop:grad:continuity}]
		Let $\mu$ and $\mu'$ denote the associated adjoint solution associated with $\rho$ and $\rho'$ respectively. Then
		\[
		\Gamma[\rho]-\Gamma[\rho']=\underbrace{\int_0^T
			\nabla\mu (\rho-\rho') \d t}_{A_1}+\underbrace{\int_0^T
			\nabla(\mu-\mu')\rho' \d t}_{A_2}.
		\]
		The first term measures the perturbation of the forward equation, while the second term measures the induced perturbation of the adjoint equation.
		
		For the first term,
		\[
		\begin{aligned}
			\left\lvert \langle A_1,\phi\rangle_x\right\rvert &\le\sum_{j=1}^d\int_0^T \left\lvert \int\,\partial_{x_j}\mu\,\phi_j\,(\rho_t-\rho_t') dx\right\rvert  dt.
		\end{aligned}
		\]
		Since Proposition~\ref{prop:adj:Existence+GrdBounds} gives $\mu\in C_{\mathfrak p}^{2}$ (one derivative in time and  two spatial derivatives are continuous), the function $\partial_{x_j}\mu\,\phi_j$ belongs to $C_b^1(\mathbb R^d)$ at any time. Therefore, by the duality formulation for Wasserstein-1:
		\[
		\left\lvert \int\partial_{x_j}\mu\,\phi_j(\rho_t-\rho_t')\,dx\right\rvert \le\operatorname{Lip}(\partial_{x_j}\mu\,\phi_j)W_1(\rho_t,\rho_t'),
		\]
		and hence
		\[
		\left\lvert \langle A_1,\phi\rangle_x\right\rvert \le C\int_0^TW_1(\rho_t,\rho_t')\d t.
		\]
		
		For the second term, let $K\supset \operatorname{supp}\phi$ be compact. Then
		\[
		\begin{aligned}
			\left\lvert \langle A_2,\phi\rangle_x\right\rvert &\le\int_0^T \|\phi\|_\infty\,\|\nabla(\mu-\mu')\|_{{L^\infty(K)}}\d t.
		\end{aligned}
		\]
		Applying Proposition~\ref{prop:adj:Existence+GrdBounds},
		\[
		\|\nabla(\mu-\mu')\|_{{L^\infty(K)}}\le C\|\mu_T-\mu_T'\|_{C^{2,\theta}_\mathfrak{p}},
		\]
		while Assumptions~\ref{assumptions:a}--\ref{assumptions:Eulerian} give
		\[
		\|\mu_T-\mu_T'\|_{C^{2,\theta}_\mathfrak{p}}
		\le
		C
		W_1(\rho_T,\rho_T').
		\]
		Consequently,
		\[
		\left\lvert 
		\langle A_2,\phi\rangle_x
		\right\rvert 
		\le
		C
		W_1(\rho_T,\rho_T').
		\]
		
		Combining the two estimates completes the proof.
	\end{proof}

	\subsection{Extension to discrete-in-time setting}~\label{subsec:gman_EM}
	In a fully discrete implementation, evaluating $\GmaN$ in~\eqref{eq:Eulerian_gradient} also requires the numerical simulation of the particle trajectories $X_{i,t}$ and the approximation of the time integral by quadrature. The preceding analysis extends naturally to these additional sources of error, allowing the trajectory-discretization and quadrature errors to be incorporated into the overall error estimate.
	
	We begin by decomposing the numerical approximation into three layers:
	\begin{align*}
		\G &= \int \nabla \mu \rho  dt\\
		&  \approx\GmaN:=\int \grd \mu^N \rho^N  dt  &&\text{particle approximation, error  } e_N^{\mathrm{E}}\\
		&\approx \overline\GmaN:=\int    \overline{\grd \mu^N} \ \overline{ \rho^N}\d t  &&\text{Riemann integral quadrature, error }  e_{\integral}^{\mathrm{E}}\\
		&\approx \overline\GmaN_{,\tau}:=\int \overline{\grd \mu^{N,\tau}}\ \overline{\rho^{N,\tau}} \d t&&\text{EM discretization, error } e_{\EM}^{\mathrm{E}}
	\end{align*}
	
	The first step produces the the particle approximations error $e_N^{\mathrm{E}}$ that arises from the empirical approximation $\rho^N$ of $\rho$ and the corresponding adjoint solution $\mu^N$  of \eqref{eq:adjoint} with terminal data \eqref{eq:adjoint:finalData}. The error is analyzed in Section~\ref{subsec:gman_main};
	
	The second step gives rise to the numerical integration error $e^{\mathrm{E}}_\integral$. For  Riemann integration in time, any time-continuous function is replaced by step functions over the time intervals $[t_k,t_{k+1})$ with $t_k=k\tau$:
	\[
	\overline {\rho^N_t} := \rho^N_{t_k} \qquad \text{and} \qquad \overline {\mu^N_t} := \mu^N_{t_k} \qquad  \text{for }\quad  t\in [t_k,t_{k+1}).
	\]
	
	The third step gives rise to $e^{\mathrm{E}}_\EM$ and is induced by the Euler--Maruyama scheme that replaces the trajectory $X_{t}$  by
	\begin{equation}\label{eq:EMScheme}
		X_{k+1}^\tau =  X_k^\tau + a(X_k^\tau) \tau  + \sqrt{2\tau } \zeta_k, \qquad X_0 \sim f_0
	\end{equation}
	for i.i.d.~standard normal $\zeta_k\sim N(0,\text{Id})$. The solution of this scheme is well known to converge strongly with order $\tau^{1/2}$ to $X_t$ \cite{kloeden1992stochastic}, and we repeat this result for convenience in \Cref{prop:EM:Rate} in the appendix. The resulting empirical measure is denoted  $\rho^{N,\tau}= \frac 1 N \sum_n \delta_{X^{n,\tau}}$ and  $$\mu^{N,\tau} \text{ solves the adjoint PDE \eqref{eq:adjoint} with final condition  } \mu^{N,\tau}(T,\cdot)
	=
	\left.
	\frac{\delta\mathcal J}{\delta\rho_T}
	\right\rvert_{\rho^{N,\tau}_{T}}.$$
	Their time-discrete counterparts are again denoted by the bar notation
	\[
	\overline {\rho^{N,\tau}_t} := \rho^{N,\tau}_{t_k} \qquad \text{and}\qquad \overline{\mu^N_t} = \mu^N_{t_k} \qquad  \text{for } t\in [t_k,t_{k+1}).
	\]
	This section is dedicated to understanding both
	\[
	e^\rmE_\integral=\mathbb E\left\lvert\langle \GmaN-\overline\GmaN\,,\phi\rangle_x\right\rvert\,,\quad\text{and}\quad e^\rmE_\EM=\mathbb E\left\lvert\langle \overline\GmaN-\overline\GmaN_{,\tau}\,,\phi\rangle_x\right\rvert\,.
	\]
	The two results are summarized in Propositions~\ref{prop:error:int} and \ref{prop:error:EM}. Together, they yield the approximation rate for the time discrete Eulerian approximation in \Cref{thm:Eulerian:timedisc}.

	\begin{theorem}\label{thm:Eulerian:timedisc}
		Under  Assumptions~\ref{assumptions:a}--\ref{assumptions:Eulerian} and \ref{assumptions:Euler:additional},  and with $\tau <1$, for any $\phi\in C_c^1(\mathbb R^d; \mathbb R^d)$, the  error of the time discrete Euler approximation satisfies, for some $N$- and $\tau$-independent constant $C$, 
		$$\mathbb E\left\lvert\langle \mathcal G-\overline\GmaN_{,\tau}\,,\phi\rangle_x\right\rvert\leq  C(\alpha(N)+\tau^{\frac 1 2}).$$
	\end{theorem}
	
	\begin{proposition}\label{prop:error:int}
		Under  Assumptions~\ref{assumptions:a}--\ref{assumptions:Eulerian} and \ref{assumptions:Euler:additional},  and with $\tau <1$, for any $\phi\in C_c^1(\mathbb R^d; \mathbb R^d)$, the integral approximation error satisfies, for some $N$- and $\tau$-independent constant $C$, 
		$$e_{\integral}=\mathbb E\left\lvert\langle \GmaN-\overline\GmaN\,,\phi\rangle_x\right\rvert \leq  C(\alpha(N)+\tau^{\frac 1 2}).$$
	\end{proposition}

	\begin{proposition}\label{prop:error:EM}
		Under  Assumptions~\ref{assumptions:a}--\ref{assumptions:Eulerian}, there exists an $N$- and $\tau$-independent $C$ such that the Euler--Maruyama error is 
		$$  e_{\EM}=\mathbb E\left\lvert \langle \overline\GmaN-\overline\GmaN_{,\tau}\,,\phi\rangle_x\right\rvert  \leq  C( \alpha(N) + \tau^{ \frac 1 2 }).$$
	\end{proposition}
	
	While the $\tau$ dependence is expected, the $N$-dependence of $e_{\mathrm{int}}$ and $e_{\mathrm{EM}}$ may seem counterintuitive from a classical textbook perspective. We should emphasize that such classical continuity approaches hides $N$-dependence in the constant, rendering these estimates less meaningful for our goal to explicitly spell out the rate in $N$ and $\tau$. We thus work with constants that  only depend on the regularity of the continuum quantities $\rho,\mu$, coefficients $a$ and the test function $\phi$.
	
	\begin{proof}[Proof of \Cref{prop:error:int}]
		First note that under Assumptions~\ref{assumptions:a} and~\ref{assumptions:Euler:additional}, $\mu\in C^1([0,T]; C^1_p(\mathbb R^d))$. Let $K\supset \operatorname{supp}\phi$ be compact. Pivoting with $\grd \mu$ shows
		\begin{align*}
			\mathbb E
			\left\lvert 
			\langle
			e_\integral^{\mathrm E},
			\phi
			\rangle_x
			\right\rvert = &\E \left[\left\lvert\left\langle \int \grd \mu^N \rho^N -  \overline{\grd \mu^N} \ \overline{\rho^N} \d t  , \phi(x)\right\rangle_x \right\rvert  \right] \\
			&= \E \left[\left\lvert \int \left( \int  \grd \mu^N\phi(x) \d \rho^N_t -  \int  \overline{\grd \mu^N}\phi(x) \d \overline{\rho}^N_t \right)\d t\right\rvert   \right]
			\\
			&= \E \bigg[\left\lvert\int \left( \int ( \grd \mu^N -\grd \mu ) \phi(x) \d \rho^N_t  +  \int  \grd \mu  \phi(x) \d \left(\rho^N_t -  \overline{\rho}^N_t\right)\right.\right. \\
			& \qquad \quad +  \left.\left. \int  ({\grd \mu}-\overline{\grd \mu})\phi(x) \d \overline{\rho}^N_t \d t  +  \int  (\overline{\grd \mu} -\overline{\grd \mu}^N)\phi(x) \d \overline{\rho}^N_t \right)\d t \right\rvert \bigg]
			\\
			&\leq \E \left[C_{\phi}\int \| \grd \mu^N_t-{\grd \mu}_t\|_{L^\infty(K)} + \| \overline{\grd \mu}^N _t-\overline{\grd \mu}_t\|_{L^\infty(K)}  \d t \right.\\
			&\qquad \quad \left. + C_{\mu\phi} \int W_1(\rho^N_t, \overline\rho^N_t) \d t+ C_{\phi} \int \|  {\grd \mu}_t -\overline{\grd \mu}_t\|_{{L^\infty(K)}}\d t\right]
			\\
			&\overset{\text{Prop.~}\ref{prop:adj:Existence+GrdBounds}}{\leq} \E \left[2 C_{\phi} C_\mu W_1(\rho^N_T, \rho_T)  + C_{\mu\phi} \int W_1(\rho^N_t, \overline\rho^N_t) \d t+ C_{\phi} CT\tau \right]
			\\
			&\overset {\text{Thm.~}\ref{thm:samplingRho},\eqref{eq:rate:SDEHoelder}}{\leq}  C'(\alpha(N)+ \tau^{1/2}+ \tau).
		\end{align*}
		
		Here, we deploy several results in Appendix~\ref{sec:gman_helper}. In the first inequality, we combine the first and fourth terms and bound them by the distance between $\mu^N$ and $\mu$. This term is further bounded using the regularity of $\mu$ according to Proposition~\ref{prop:adj:Existence+GrdBounds}. The second term is translated to the Wasserstein bound by Proposition~\ref{prop:adj:Existence+GrdBounds}, using the fact that $\nabla\mu$ is bounded. Its rate $\tau^{1/2}$ follows from \Cref{cor:rate:time_particles}. The third term is the true discretization-in-time term and is controlled by the temporal regularity of $\mu$ from Proposition~\ref{prop:adj:Existence+GrdBounds}.
	\end{proof}
	
	\begin{proof}[Proof of \Cref{prop:error:EM}]
		Again, pivoting with $\grd \mu$ gives :
		\begin{align*}
			e_{\EM} 
			&=\E\bigg[\left\lvert  \bigg\langle \int (\overline{\grd \mu^N} -\overline{\grd\mu} ) \overline{\rho^N} \d t + \int\overline{\grd\mu} ( \overline{\rho^N}- \overline{\rho^{N,\tau}})\d t\right.
			\\
			&\qquad \quad \left.\left.\left.+\int  (\overline{\grd\mu}- \overline{\grd \mu^{N,\tau}})\overline{\rho^{N,\tau}} \d t, \phi\right\rangle_x\right\rvert \right] 
			\\
			&\leq \E\left [C_\phi \int\|\overline{ \grd \mu}^N_t- \overline{\grd\mu}_t\|_{L^\infty(K)} + \| \overline{\grd\mu}_t- \overline{\grd \mu}^{N,\tau}_t\|_{L^\infty(K)} \d t \right.\\
			& \qquad \quad \left. +  C_{\mu\phi} \int W_1(\overline{\rho}^N_t, \overline{\rho}^{N,\tau}_t)\d t \right]
			\\
			&\overset {\text{Prop. }\ref{prop:adj:Existence+GrdBounds}}\leq  \E\left[ C_\phi C_\mu\left(W_1(\rho_T,\rho^N_T) +W_1(\rho_T,\rho^{N,\tau}_T) \right)+ C_{\mu\phi}\int W_1(\overline\rho^{N}_t,\overline\rho^{N,\tau}_t)\d t\right]\\
			&\overset{\text{Thm.~\ref{thm:samplingRho}}}\leq C_\phi C_\mu(2C \alpha(N) + C_{EM}\tau^{1/2})  + C_{\mu\phi}TC_{EM}\tau^{1/2}.
		\end{align*}
		In the last line, the first term and one part of the second term are bounded by the sampling error in Equation~\eqref{eq:alphaN}. The strong rate of convergence of the Euler--Maruyama scheme \cite{kloeden1992stochastic}  then allows us to control the distance between $\overline\rho^N_t$ and $\overline\rho^{N,\tau}_t$, for every $0\leq t\leq T$:
		\begin{align*} 
			\E[W_1(\overline\rho^{N}_t,\overline\rho^{N,\tau}_t)] &\leq \E\left[\frac 1N \sum_{i=1}^N\left\lvert X_{i,t} - X_{i,t}^\tau\right\rvert\right] \leq C_{EM}\tau^{1/2}.
		\end{align*}
	\end{proof}

	\section{Lagrangian Formulation}\label{sec:Lagrangian}

	We now turn to the discretize-then-optimize (DTO) formulation. We first discretize the forward Fokker--Planck equation~\eqref{eq:FP} by the particle system~\eqref{eq:particle} and then differentiate the resulting finite-dimensional dynamics to obtain the particle adjoints and the corresponding gradient. In complete analogy with the derivation in \Cref{thm:cont_correspondence}, the discretized gradient can be represented as the vector-valued measure 
	\begin{equation}\label{eq:Lagrangial_gradient}
		\GmiN = \frac1N \sum_{i=1}^N \int_0^T Y_{i,t}\,\delta_{x-X_{i,t}}\,\d t,
	\end{equation} 
	where the adjoint particles $Y_i$ satisfy~\eqref{eq:adjoint_particle}, with terminal conditions determined by the empirical measure $\rho_T^N$. 
	
	At first glance, $\GmiN$ resembles a standard Monte Carlo approximation of the continuum representation \[ \G = \mathbb E \left[ \int_0^T Y_t\,\delta_{x-X_t}\,\d t \right], \] with the expectation replaced by an empirical average. Our objective in this section is to quantify the resulting error \[ e_N^{\mathrm L} := \G-\GmiN. \] 
	
	This error is not, however, a classical Monte Carlo error. For each $t$, the particles $X_{i,t}$ are independent samples from the law $\rho_t$ of $X_t$; equivalently, \[ \mathbb E[\delta_{X_{i,t}}]=\rho_t \] in the weak sense. The state approximation can therefore be analyzed using standard Monte Carlo arguments. The approximation of the pathwise adjoint $Y_t$ by $Y_{i,t}$ is more subtle. The continuum adjoint has terminal value \[ Y_T = \nabla \left. \frac{\delta\mathcal J}{\delta\rho} \right\rvert_{\rho_T}(X_T), \] whereas the particle adjoint is initialized using \[ Y_{i,T} = \nabla \left. \frac{\delta\mathcal J}{\delta\rho} \right\rvert_{\rho_T^N}(X_{i,T}). \] 
	
	Although 
	$ \mathbb E[\rho_T^N]=\rho_T$
	as measures, the nonlinear dependence of $\delta\mathcal J/\delta\rho$ on its measure argument generally prevents $Y_{i,t}$ from being an unbiased approximation of $Y_t$. Moreover, even if the two factors were individually unbiased, the product $Y_{i,t}\delta_{x-X_{i,t}}$ would not generally be unbiased because the adjoint and the particle position are correlated. 
	
	Nevertheless, the empirical measure $\rho_T^N$ converges to $\rho_T$, and the particle adjoint depends continuously on its terminal data. One therefore expects $\GmiN$ to converge to $\G$. The central question is quantitative: to what extent, if at all, does this additional dependence on the empirical terminal measure degrade the classical $N^{-1/2}$ Monte Carlo rate? 
	
	In~\Cref{subsec:gmin_main}, we decompose the error into a classical Monte Carlo component and a stability error caused by replacing $\rho_T$ with $\rho_T^N$, and derive a quantitative convergence estimate for $\GmiN$. In~\Cref{subsec:gmin_EM}, we extend the analysis to the fully discrete setting by incorporating time-quadrature and Euler--Maruyama discretization errors.

	\subsection{Main Theorem}\label{subsec:gmin_main}
	As in the Eulerian setting, $\GmiN$ is a vector-valued measure, so the error must be understood in the weak sense and evaluated against a test function $\phi\in C^1_c(\mathbb R^d;\mathbb R^d)$.

	\begin{theorem}\label{thm:GmiN}
		Under Assumptions~\ref{assumptions:a}, \ref{assumptions:rho0},~\ref{assumptions:Lagrangian}, the Lagrangian approximation $\mathcal G^{\mathrm{L}}_N$ from \eqref{eq:Lagrangial_gradient} converges weakly to the continuum gradient $\mathcal G$. More precisely, for every $\phi\in C_c^1(\mathbb R^d;\mathbb R^d)$, there exists a constant $C$ independent of $N$ with
		\begin{equation}\label{eq:convergence_GmiN}
			\mathbb E
			\left\lvert 
			\langle
			\mathcal G-\mathcal G_N^{\mathrm L},
			\phi
			\rangle_x
			\right\rvert  
			\leq 
			C \alpha(N),
		\end{equation}
		where $\alpha(N)$ is defined in~\eqref{eq:alphaN}.
	\end{theorem}
	
	As discussed, one major difficulty in analyzing $e^\rmL_N$ is the bias between $Y_{i,t}$ and $Y_t$. To handle this, we introduce the intermediate approximation
	\begin{equation}
		\widetilde {\mathcal G}_N^{\mathrm L}:= \frac1N\sum_{i=1}^N\int_0^T\widetilde Y_{i,t}\,\delta_{x-X_{i,t}}
		\d t\,,
	\end{equation} 
	with simulated $X_{i,t}$ and a modified adjoint:
	\begin{equation}
		\text{$\widetilde Y_i$ satisfy \eqref{eq:adjoint_particle} with final data } \widetilde Y_{i,T} = 
		\nabla
		\left.
		\frac{\delta\mathcal J}{\delta\rho}
		\right\rvert_{\rho_T}
		(X_{i,T})\,.
		\label{eq:tildeGL}
	\end{equation} 
	Then $(X_i,\tilde Y_i)$ are i.i.d. copies of $(X,Y)$. This  permits a decomposition into the Monte Carlo error and a bias term:
	\begin{align*}
		e^\rmL_N=\mathcal G-\widetilde{\mathcal G}_N^{\mathrm L} + \widetilde{\mathcal G}_N^{\mathrm L}-\mathcal G_N^{\mathrm L}
	\end{align*}
	so that
	\begin{align*}
		\mathbb E
		\left\lvert 
		\langle
		e^\rmL_N,
		\phi
		\rangle_x
		\right\rvert  
		\leq 
		\mathbb E
		\left\lvert 
		\langle
		\mathcal G-\widetilde{\mathcal G}_N^{\mathrm L},
		\phi
		\rangle_x
		\right\rvert   + \mathbb E
		\left\lvert 
		\langle
		\widetilde{\mathcal G}_N^{\mathrm L}-\mathcal G_N^{\mathrm L},
		\phi
		\rangle_x
		\right\rvert   = B_1 + B_2.
	\end{align*}
	The proof then proceeds to analyze these two terms respectively.
	
	Before turning to the proof, we summarize several   properties of $\widetilde{Y}$:
	\begin{lemma}\label{lem:tildeY}
		Under Assumptions~\ref{assumptions:a}, \ref{assumptions:rho0},~\ref{assumptions:Lagrangian}, consider $\widetilde Y_i$ as in \eqref{eq:tildeGL}. 
		\begin{enumerate}[(a)] 
			\item \label{lem:tildeY:bound} For the constant $C_Y$ from \Cref{cor:adjoint:MomentBd+Lipschitz},
			$$\mathbb E\left\lvert\widetilde Y_{i,t}\right\rvert^2\leq C_{Y}.$$
			\item \label{lem:tildeY:samplingrate} There exists a constant $C_{\widetilde Y}$, independent of $N$, such that for all $t\in [0,T]$
			$$\mathbb E\left\lvert\widetilde Y_{i,t}-Y_{i,t}\right\rvert \leq C_Y' C_J \mathbb E[W_1(\rho_T, \rho_T^N)] \leq C_{\widetilde Y} \alpha(N),$$
			with $C_Y'$ from \Cref{prop:adjParticles:Existence+Bounds} .
		\end{enumerate}   
	\end{lemma}
	We defer the proof to Appendix. Note that $C_{\widetilde Y}$  depends on the second moment of $\rho_T$ through \Cref{thm:samplingRho}.
	\begin{proof}[Proof of~\Cref{thm:GmiN}]
		Note that the modified adjoint introduces an i.i.d. sample, thus $(X_i,\widetilde Y_i)\sim (X,Y)$  i.i.d.. Define $Z:= \int_0^T Y_t \phi(X_t) \d t$, then $\widetilde Z_i = \int_0^T\widetilde Y_{i,t}\,\phi(X_{i,t}) \d t$ are also i.i.d. copies of the scalar random variable $Z$. Consequently, the first term can be bounded by standard Monte Carlo arguments:
		\begin{align*}
			B_1= 
			\mathbb E
			\left\lvert 
			\langle
			\mathcal G-\widetilde{\mathcal G}_N^{\mathrm L},
			\phi
			\rangle_x
			\right\rvert   = \mathbb E
			\left\lvert \mathbb E\left[Z\right] -\frac1N\sum_{i=1}^N \widetilde Z_i
			\right\rvert  \leq \left(\frac{\operatorname{Var}(Z)}{N}\right)^{\frac 1 2 } = C_{\phi}C_Y^{\frac1 2} T N^{-\frac 1 2 }.
		\end{align*}
		In the inequality we used Cauchy--Schwarz, and the finiteness of the variance is a consequence of \Cref{cor:adjoint:MomentBd+Lipschitz}
		
		For the second term, we note that $\widetilde{Y}_{i,t}$ and ${Y}_{i,t}$ satisfy the same dynamics but have different terminal conditions. The estimate then comes down to exploiting the stability of this adjoint process. According to \Cref{prop:adjParticles:Existence+Bounds}
		\begin{align*}
			B_2  =  \
			\mathbb E
			\left\lvert 
			\langle
			\widetilde{\mathcal G}_N^{\mathrm L}-\mathcal G_N^{\mathrm L},
			\phi
			\rangle_x
			\right\rvert   
			= & \ \mathbb E
			\left\lvert  \frac1N\sum_{i=1}^N\int_0^T(\widetilde Y_{i,t} - Y_{i,t})\,\phi(X_{i,t})
			\d t
			\right\rvert  
			\\
			\leq &  \ \frac {C_\phi} N\sum_{i=1}^N\int_0^T\mathbb E
			\left\lvert  \widetilde Y_{i,t} - Y_{i,t}
			\right\rvert  \d t \\
			\leq & \ C_\phi C_{\widetilde Y} T \alpha(N) \quad = \quad C \  \alpha(N)
		\end{align*}
		for as $N$-independent constant $C$.
		In the last line, we applied \Cref{lem:tildeY}\ref{lem:tildeY:samplingrate}.
	\end{proof}

	\subsection{Discrete-in-time setting}\label{subsec:gmin_EM}
	As discussed in the Eulerian setting, we can extend the results above to the fully discrete setting. Three layers of approximations are deployed.
	
	\begin{align*}
		\G &= \mathbb E
		\left[
		\int_0^T
		Y_t\delta_{x-X_t}\d t
		\right]\\
		& \approx\GmiN=\frac1N\sum_{i=1}^N\int_0^TY_{i,t}\,\delta_{x-X_{i,t}}
		\d t,&&\text{particle approximation, } e^\rmL_N\\
		&\approx \overline\GmiN:=\frac1N\sum_{i=1}^N\int_0^T\overline{Y_{i,t}}\,\delta_{x-\overline{X_{i,t}}}
		\d t  &&\text{Riemann integral quadrature, } e^\rmL_\integral \\
		&\approx \overline\GmiN_{,\tau}:=\frac1N\sum_{i=1}^N\int_0^T\overline{Y_{i,t}^\tau}\,\delta_{x-\overline{X_{i,t}^\tau}}   \d t  &&\text{EM discretization, } e^\rmL_\EM\,.
	\end{align*}
	The first layer of error $e^\rmL_N$ has already been analyzed in~\Cref{subsec:gmin_main}. We now analyze the integral error and EM discretization error, where we use the notation 
	$$\overline{Y_{i,t}} = Y_{i,t_k},  \quad  \overline{X_{i,t}} = X_{i,t_k}, \qquad \overline{Y_{i,t}^\tau} = Y_{i,k}^\tau,  \quad  \overline{X_{i,t}^\tau} = X_{i,k}^\tau, \quad \text{for} \quad t\in [t_k, t_{k+1}).$$
	Here we let the adjoint ODE \eqref{eq:adjoint_particle} be solved according to the pathwise backward Euler scheme:
	\begin{equation}\label{eq:EulerScheme:Adjoint}
		Y_{i,k}^\tau  =  Y_{i,k+1}^\tau + \tau  Y_{i,k+1}^\tau \grd a(X_{i,k}^\tau)
	\end{equation}
	with final condition
	\begin{equation}\label{eq:EulerScheme:Adjoint:final}
		Y_{i,K}^\tau =  \grd \left.\frac {\delta \mathcal J}{\delta \rho_T}\right\rvert_{\rho_T^{N,\tau}}(X_{i,K}^\tau)\,.
	\end{equation}
	The error in measure argument of the final condition \eqref{eq:EulerScheme:Adjoint:final}, relative to~\eqref{eq:adjoint_particle}, 
	is controlled  by  Assumption~\ref{assumptions:Lagrangian}, whereas the shift in the evaluation point requires additional spatial regularity of $\grd \left.\frac{\delta \mathcal J}{\delta \rho_T}\right\rvert_{\rho_T}$.
	\begin{assumption}\hfill\\
		\begin{enumerate}
			\item[(A5)]\makeatletter\protected@edef\@currentlabel{(A5)}\makeatother
			\label{assumptions:JLipschity} For every $\nu\in \mathcal P_q$, let $\grd \left.\frac{\delta \mathcal J}{\delta \rho_T}\right\rvert_{\nu}$ be Lipschitz continuous: there exists an $L_J = L_J(\nu)$ such that for all $x,y\in \mathbb R^d$ 
			$$\left\lvert \grd \left.\frac{\delta \mathcal J}{\delta \rho_T}\right\rvert_{\nu}(x) - \grd \left.\frac{\delta \mathcal J}{\delta \rho_T}\right\rvert_{\nu}(y)\right\rvert  \leq L_J\left\lvert x-y\right\rvert .$$
		\end{enumerate}
	\end{assumption}
	
	\begin{theorem}\label{thm:Lagrange:timedisc}
		Consider Assumptions~\ref{assumptions:a}, ~\ref{assumptions:rho0},~\ref{assumptions:Lagrangian},~\ref{assumptions:JLipschity}, and with $a\in W^{2,\infty}(\mathbb R^{d}; \mathbb R^d)$. Then, for $\tau \leq 1$ and for any fixed test function $\phi\in C^1_c(\mathbb R^d; \mathbb R^d)$, there exists a  constant $C$ independent of $N,\tau$ such that 
		$$
		\E\left\lvert\langle \mathcal G -  \overline\GmiN_{,\tau}, \phi\rangle_x\right\rvert  \leq  C\left(\alpha(N) + \tau^{1/2}\right).
		$$
	\end{theorem}
	The  two propositions below split the total error and analyze the terms 
	\[
	e^\rmL_\integral=\mathbb E\left\lvert\langle \GmiN-\overline\GmiN\,,\phi\rangle_x\right\rvert\,,\quad\text{and}\quad e^\rmL_\EM=\mathbb E\left\lvert\langle \overline\GmiN-\overline\GmiN_{,\tau}\,,\phi\rangle_x\right\rvert\,.
	\]
	Together with \Cref{thm:GmiN}, they prove \Cref{thm:Lagrange:timedisc}.
	
	\begin{proposition}\label{prop:error:Lagrange:int}
		Let Assumptions~\ref{assumptions:a}, ~\ref{assumptions:rho0},~\ref{assumptions:Lagrangian} hold. For  $\tau \leq 1$, there exists a constant $C$ independent of $N,\tau$, such that  
		$$
		e^{\mathrm{L}}_{\integral} = \mathbb E\left\lvert\langle \GmiN - \overline{\GmiN}, \phi\rangle_x\right\rvert \leq C\left(\alpha(N) + \tau^{1/2}\right).
		$$
	\end{proposition}
	
	\begin{proof}We again pivot against the mean field particles $\widetilde Y_i$ from \eqref{eq:tildeGL}:
		\begin{align*}
			\langle \GmiN - \overline{\GmiN}, \phi\rangle_x =&  \frac 1 N \sum_i \left\langle\int  Y_{i,t} \delta_{x-X_{i,t}} -  \overline{Y_{i,t} }\delta_{x-\overline{X_{i,t}}} \d t\ , \ \phi(x)\right\rangle_x \\
			=&\frac 1 N \sum_i \int  Y_{i,t} \phi(X_{i,t})  +(1-1)  \widetilde Y_{i,t} \phi(X_{i,t})  +(1-1)  \overline{\widetilde Y_{i,t}} {\phi(X_{i,t})} \\
			&\hspace{1.5cm} +(1-1) \overline{\widetilde Y_{i,t}} \overline{\phi(X_{i,t})} - \overline{Y_{i,t} }\phi(\overline{X_{i,t}}) \d t
		\end{align*}
		By exchangeability of the particle pairs $(X_i,Y_i)$, the mean absolute error $\mathbb E\left\lvert \langle \GmiN - \overline{\GmiN}, \phi\rangle_x\right\rvert$ can be upper bounded by the one-particle mean error, without loss of generality for $i=1$, and the triangle inequality gives
		\begin{align*}
			e^{\mathrm{L}}_{\integral}  \leq &  \int  \mathbb E[\lvert Y_{1,t}-{\widetilde Y_{1,t}}\rvert\left\lvert\phi(X_{1,t})\right\rvert] + \mathbb E[\lvert {\widetilde Y_{1,t}} - \overline{\widetilde Y_{1,t}}\rvert\left\lvert \phi({X_{1,t}})\right\rvert] 
			\\
			&\qquad 
			+ \mathbb E[\lvert \overline{\widetilde Y_{1,t}}||\phi(X_{1,t}) - \phi(\overline{X_{1,t}})\rvert]
			+ \mathbb E[\lvert \overline{\widetilde Y_{1,t}} -\overline{ Y_{1,t}} || \phi(\overline{X_{1,t}})\rvert] \d t \\
			=  & \ 
			B_1 + B_2+B_3 + B_4.
		\end{align*}
		By boundedness of $\phi$, the pivoting terms $B_1$ and $B_4$ are dominated by the sampling and stability errors in \Cref{lem:tildeY}\ref{lem:tildeY:samplingrate}, and there exists an $N$-independent constant $C$ for which
		$$B_1+B_4  \quad \leq  \quad C_\phi C_{\widetilde Y} T \alpha(N) \quad = \quad  C \alpha(N).$$
		Insert the evolution equation~\eqref{eq:adjoint_particle} and the standard ODE error estimate, together with \Cref{lem:tildeY}\ref{lem:tildeY:bound}, to obtain
		\begin{align*}
			B_2 
			= \int   \mathbb E\left[\left\lvert \int_{\overline t}^t -\grd a(X_{1,s})^\top\widetilde Y_{1,s} \d s\right\rvert\left\lvert \phi({X_{1,t}})\right\rvert\right] \d t
			\leq C_{\phi}\|\grd a\|_\infty C_Y^{1/2} T \tau =C' \tau.
		\end{align*}
		under the required regularity assumptions~\ref{assumptions:a}, \ref{assumptions:Lagrangian}. Similarly, apply Lipschitzness of $\phi$ and  H\"older's inequality
		\begin{align*}
			B_3
			\leq & \ C_\phi\int \mathbb E[\lvert \overline{\widetilde Y_{1,t}}\rvert\lvert X_{1,t} -\overline{X_{1,t}}\rvert]\d t \leq C_\phi\int \sqrt{\mathbb E[\lvert \overline{\widetilde Y_{1,t}}\rvert^2]}\sqrt{\mathbb E[\lvert X_{1,t} -\overline{X_{1,t}}\rvert^2]}\d t\\
			\leq & \ C_{\phi} C_Y^{1/2}  C_X^{1/2} T\tau^{1/2} \quad = \quad  C'' \tau^{1/2} .
		\end{align*}
		where we used the  boundedness of moments of $\widetilde Y$ in \Cref{lem:tildeY}\ref{lem:tildeY:bound} and the temporal regularity of the particle dynamics in  \Cref{prop:Particles:Existence+Bounds} with constant $C_X$.
	\end{proof}

	\begin{proposition}\label{prop:error:Lagrange:EM}
		Let Assumptions~\ref{assumptions:a}, ~\ref{assumptions:rho0},~\ref{assumptions:Lagrangian},~\ref{assumptions:JLipschity} and with $a\in W^{2,\infty}(\mathbb R^d)$  hold. For $\tau \leq 1$,  there exists a constant $C$ independent of $N,\tau$, such that 
		$$
		e^{\mathrm{L}}_{\EM} = \mathbb E\left\lvert \langle  \overline{\GmiN}-  \overline{\GmiN}_{,\tau}, \phi\rangle_x\right\rvert \leq C\left(\alpha(N) + \tau^{1/2}\right) .
		$$
	\end{proposition}

	Before  bounding the Euler--Maruyama error, we collect some classical properties of the Euler scheme. 
	
	\begin{lemma}
		[Accuracy and stability of Euler scheme] \label{lem:EulerScheme:Stab+Acc}
		Let  Assumptions~\ref{assumptions:a}, ~\ref{assumptions:rho0},~\ref{assumptions:Lagrangian} hold, and additionally assume $a\in W^{2,\infty}(\mathbb R^d)$.  Fix $i$ and let $X_{i,k}^\tau$ be a solution to \eqref{eq:EMScheme}. 
		\begin{enumerate}[(a)]
			\item \label{lem:EulerScheme:Stability} (stability) There exists a constant $C_{Y^\tau}$, depending only on $a$ and $T$, such that the solution  $\widetilde Y_{i,k}^\tau$ to the Euler scheme \eqref{eq:EulerScheme:Adjoint} with final condition 
			$\widetilde Y_{i,K}^\tau =  \zeta$
			satisfies 
			$$\left\lvert\widetilde Y_{i,k}^\tau\right\rvert\leq C_{\widetilde Y^\tau}\left\lvert \zeta\right\rvert.$$
			
			\item  \label{lem:EulerScheme:Accuracy} (accuracy) Let $\widetilde Y_{i,t}$ be the solution of \eqref{eq:tildeGL} and $\widetilde Y_{i,k}^\tau$ be the solution of \eqref{eq:EulerScheme:Adjoint} with final condition $\zeta = \grd \frac{\delta \mathcal J}{\delta \rho_T}\large\mid_{\rho_T}(X_{i,T})$. There exists a constant $C'_{Y^\tau}$, depending only on $a$ and $T$ and the uniform second moment bound on $\widetilde Y$, such that 
			$$\mathbb E\left\lvert \widetilde Y_{i,t_k}-\widetilde Y_{i,k}^\tau\right\rvert\leq C'_{\widetilde Y^\tau} \tau^{1/2}.$$
		\end{enumerate}
		
	\end{lemma}
	
	\begin{proof}[Proof of \Cref{prop:error:Lagrange:EM}]
		\begin{align*}
			\langle \overline{\GmiN}-  \overline{\GmiN}_{,\tau}, \phi\rangle_x =&  \frac 1 N \sum_i \left\langle\int    \overline{Y_{i,t} }\delta_{x-\overline{X_{i,t}}}  - \overline{Y_{i,t}^\tau  }\delta_{x-\overline{X_{i,t}^\tau}}\d t\ , \ \phi(x)\right\rangle_x \\
			=&\frac 1 N \sum_i \int  \overline{Y_{i,t} }\phi(\overline{X_{i,t}}) +(1-1) \overline{\widetilde Y_{i,t} }\phi(\overline{X_{i,t}}) +(1-1) \overline{\widetilde Y_{i,t}^\tau }\phi(\overline{X_{i,t}})\\
			&\hspace{1.2cm}+(1-1) \overline{\widetilde Y_{i,t}^\tau }\phi(\overline{X_{i,t}^\tau})-  \overline{Y_{i,t}^\tau }\phi(\overline{X_{i,t}^\tau}) \d t
		\end{align*}
		Again, the mean average error decomposes 
		$$
		e^{\mathrm{L}}_{\EM}  \leq C_1 + C_2 + C_3 + C_4
		$$
		and the individual error components can be written in terms of  one-particle errors by exchangeability. The pivoting errors are governed by the sampling rate from \Cref{lem:tildeY}\ref{lem:tildeY:samplingrate}
		\begin{align*}
			C_1 = & \ \E\left\lvert \int (\overline{Y_{1,t} } - \overline{\widetilde Y_{1,t} })\phi(\overline{X_{1,t}})\d t \right\rvert 
			\leq  \ C_\phi \int\E\left\lvert \overline{Y_{1,t} } - \overline{\widetilde Y_{1,t} }\right\rvert  \d t  
			\leq C_\phi C_{\widetilde Y} T \alpha(N), 
		\end{align*}
		and the stability of the Euler scheme, see \Cref{lem:EulerScheme:Stab+Acc}\ref{lem:EulerScheme:Stability}, in combination with the Euler-Maruyama rate \eqref{eq:EM:Rate}:
		\begin{align*}
			C_4 = & \   \E\left\lvert \int(\overline{\widetilde Y_{1,t}^\tau }-  \overline{Y_{1,t}^\tau })\phi(\overline{X_{1,t}^\tau}) \d t \right\rvert  
			\leq  \ C_\phi \int \E\left\lvert \overline{\widetilde Y_{1,t}^\tau }-  \overline{Y_{1,t}^\tau }\right\rvert \d t  \\
			\leq &\  C_\phi C_{\widetilde Y^\tau} T \E\left\lvert \grd \left.\frac{\delta \mathcal J}{\delta \rho_T}\right\rvert_{\rho_T}(X_{1,T}) +(1-1)\left.\frac{\delta \mathcal J}{\delta \rho_T}\right\rvert_{\rho_T}(X_{1,T}^\tau) - \left.\frac{\delta \mathcal J}{\delta \rho_T}\right\rvert_{\rho_T^{N,\tau}}(X_{1,T}^\tau)\right\rvert  \\
			\overset{\text{\ref{assumptions:JLipschity},\ref{assumptions:Lagrangian}}}\leq & \ C_\phi C_{\widetilde Y^\tau} T\left(L_J(\rho_T) \E \left\lvert X_{1,T}-X_{1,T}^\tau\right\rvert +C_J \E[W_1({\rho_T},{\rho_T^{N,\tau}})]  \right)\\
			\leq & \  C_\phi C_{\widetilde Y^\tau} T\left(L_J(\rho_T) C_{EM} \tau^{1/2} +C_J (C_{\rho}\alpha(N) + C_{EM}\tau^{1/2}) \right),
		\end{align*}
		where the last line combined \Cref{prop:EM:Rate,thm:samplingRho,corr:EM:RateWasserstein}.
		
		The remaining parts describe the temporal discretization error.
		The ODE discretization error is treated in \Cref{lem:EulerScheme:Stab+Acc}\ref{lem:EulerScheme:Accuracy}:
		\begin{align*}
			C_2= & \ \E\left\lvert \int (\overline{\widetilde Y_{1,t} } - \overline{\widetilde Y_{1,t}^\tau })\phi(\overline{X_{1,t}})\d t \right\rvert \leq C_\phi \int \E\left\lvert \overline{\widetilde Y_{1,t} } - \overline{\widetilde Y_{1,t}^\tau }\right\rvert \d t \leq C_{\phi}C'_{\widetilde Y^\tau} T \tau^{1/2},
		\end{align*}
		where the constants only depend on $a,T$ and the uniform second moment bound on $\widetilde Y_{1,t}$ from \Cref{cor:adjoint:MomentBd+Lipschitz} and are therefore independent of $N,\tau$.
		For the SDE discretization error in $C_3$, we leverage Lipschitzness of $\phi$ and the Euler--Maruyama error on the SDE
		\begin{align*}
			C_3= & \ \E\left\lvert \int  \overline{\widetilde Y_{1,t}^\tau }(\phi(\overline{X_{1,t}})- \phi(\overline{X_{1,t}^\tau}))\d t \right\rvert \leq C_\phi \int \E\left[\lvert \overline{\widetilde Y_{1,t}^\tau }\rvert \lvert \overline{X_{1,t}}- \overline{X_{1,t}^\tau}\rvert \right]\d t \\
			\leq &C_\phi C'_{\widetilde Y^\tau} C_{EM}\tau^{1/2},
		\end{align*}
		together with  the uniform in $t$ moment bounds on $\widetilde Y_1^\tau$ that arise as a consequence of \Cref{lem:EulerScheme:Stab+Acc}\ref{lem:EulerScheme:Stability} and Assumption~\ref{assumptions:Lagrangian}  $$\E\left\lvert \overline{\widetilde Y_{1,t}^\tau }\right\rvert ^2\leq  C_{\widetilde Y^\tau}^2\E\left\lvert \grd \left. \frac{\delta \mathcal J}{\delta \rho_T}\right\rvert _{\rho_T}(X_{1,T})\right\rvert^2 =: C'_{\widetilde Y^\tau}. $$
		Clearly,  $C'_{\widetilde Y^\tau}$ depends on $\rho_T$, but not $N$ or $\tau$.
		
	\end{proof}
	
	\section{Numerical Examples}

	In the following, we numerically verify the convergence of the Eulerian and  Lagrangian particle approximations~\eqref{eq:Eulerian_gradient} and~\eqref{eq:Lagrangial_gradient} to the continuum gradient~\eqref{eq:Gradient_two_perspectives}; see \Cref{thm:GmiN,thm:GmaN} and \Cref{prop:error:EM,prop:error:int,prop:error:Lagrange:int,prop:error:Lagrange:EM}.

	\subsection{Test Problem}
	We study the following test case in dimensions $d=1$ and $d=2$:
	
	\paragraph{Drift.} The advection coefficient $a$ is fixed: For fixed $\delta>0$, we choose the form
	$$a(x) =  \grd V(x) \qquad \text{with} \qquad V(x)=\delta^2\big(\sqrt{1+\left\lvert x\right\rvert ^2/\delta^2}-1\big), $$
	motivated by a  pseudo-Huber potential term in a gradient flow setting.
	This potential is a classical surrogate for the $L^1$ loss that appears frequently in the context of robust training for machine learning models and 
	sparse recovery in imaging or physics-based inverse problems \cite{song2024improved,barron2019general,charbonnier1997deterministic,kawakami2023approximate}.

	\paragraph{Objective functions.} 
	We study two objective functionals of the final-time density $\rho_T$ at time horizon $T=1$:
	\begin{enumerate}
		\item[(J1)] \emph{Linear objective with Gaussian weight,}
		\begin{equation}\tag{$J1$}\label{eq:linearObjective}
			\mathcal J[\rho_T]=\int_{\mathbb R^d} \psi(x) \d\rho_T(x),
			\qquad
			\psi(x) = e^{-\frac{\left\lvert x\right\rvert^2}{2}} .
		\end{equation}
		In this case, the adjoint final condition $\left.\frac{\delta\mathcal J}{\delta\rho_T}\right\rvert_{\nu}=\psi$ is constant in $\nu$ and, thus, the adjoint and pathwise-adjoint approximations coincide exactly with their continuum counterparts $\mu^N = \mu$ and   $Y^i = \widetilde {Y^i}$.
		\item[(J2)] \emph{Kernel (MMD) energy with a Gaussian kernel,}
		\begin{equation}\tag{$J2$}\label{eq:MMDObjective}
			\mathcal J[\rho_T]
			=
			\frac12\iint_{\mathbb R^d\times\mathbb R^d} K(x,y)\d\rho_T(x)\d\rho_T(y),
			\qquad
			K(x,y) = e^{-\frac{\left\lvert x-y\right\rvert^2}{2}} .
		\end{equation}
		In this case, the terminal condition $\left.\frac{\delta\mathcal J}{\delta\rho}\right\rvert_{\nu}=K\ast \nu$ genuinely depends on the measure.  This satisfies all points in \Cref{assumptions}, as shown in \Cref{rem:assumptions:linear+MMD}.
	\end{enumerate}
	
	\paragraph{Initial datum.} We choose a bimodal Gaussian mixture as initial datum
	\begin{equation*}
		\rho_0(x)
		=
		\frac{1}{2\,(2\pi\sigma_0^2)^{d/2}}
		\left(
		e^{-\frac{\left\lvert x-2\cdot\mathds 1\right\rvert^2}{2\sigma_0^2}}
		+
		e^{-\frac{\left\lvert x+2\cdot\mathds 1\right\rvert^2}{2\sigma_0^2}}
		\right),
		\qquad \sigma_0 = 0.35,
	\end{equation*}
	with $\mathds 1 = (1,\dots,1)\in\mathbb R^d$, so that $\rho_0\in\mathcal
	P_q(\mathbb R^d)$ for every $q$.
	
	\paragraph{Test field.} All errors are reported  against
	the fixed compactly supported test field $\phi 
	\in C^1_c(\mathbb R^d;\mathbb R^d)$ constructed from the two bump functions with width $w=2$, unless stated differently:
	\[
	\phi_w(x) = 3\,\Phi\left(\frac{x-2}w\right)-3\,\Phi\left(\frac{x+3}w\right),
	\qquad
	\Phi(y) = e^{-\frac{1}{1-\left\lvert y\right\rvert^2}}\ \mathds 1_{\{\left\lvert y\right\rvert<1\}} .
	\]
	
	Over time, we expect the two initial modes to be pushed towards the center by the drift term, and eventually merge.

	\subsection{Particle approximation}
	
	The connection between the Fokker--Planck equation \eqref{eq:FP} and the particle dynamics \eqref{eq:process} is well established through $\rho_t = \operatorname{Law}(X_t)$, and the approximation with the empirical density $\rho^N_t$ is classical. We visualize this behavior in the following figures: In dimension $d=1$, \Cref{fig:forward_1D} shows agreement of the densities over the full propagation period $[0,T]$. \Cref{fig:final_1D} provides a more detailed comparison at the final time $T=1$  and its consequences for the terminal conditions $\mu_T$ and $\mu_T^N$ of the adjoint equation, corresponding to the MMD objective~\eqref{eq:MMDObjective}. In dimension $d=2$, we only show the final densities $\rho_T, \rho_T^N$ and the resulting adjoint terminal conditions $\mu_T,\mu_T^N$, see \Cref{fig:forward_2D}.
	
	The following computational settings were used to generate the plots: 
	\begin{enumerate}
		\item[(S1)] \makeatletter\protected@edef\@currentlabel{(S1)}\makeatother \label{itm:PDEsetting} The PDE~\eqref{eq:FP} is solved  on the domain $x\in [-8,8]^d$ with no-flux boundary conditions by a finite-volume Crank--Nicolson scheme on a space-time grid with time step size $\tau^{PDE} = 1/n_t= 1/20000$ and a dimension-dependent spatial step size of $1/n_x = d/400$ in each direction. 
		\item[(S2)]\makeatletter\protected@edef\@currentlabel{(J2)}\makeatother   $N$  particle trajectories are computed by an Euler--Maruyama scheme~\eqref{eq:EMScheme} with time step $\tau = 2^{-8}$, with $N\in \{200,2000\}$ in dimension $d=1$ and $N\in \{200,10000\}$ for $d=2$.
	\end{enumerate}
	
	\begin{figure}[h]
		\centering
		\includegraphics[width=.33\linewidth]{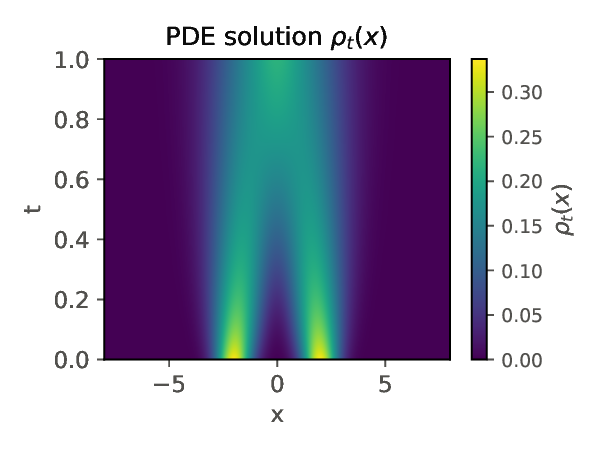} ~\hspace{-.5cm}~ \includegraphics[width=.33\linewidth]{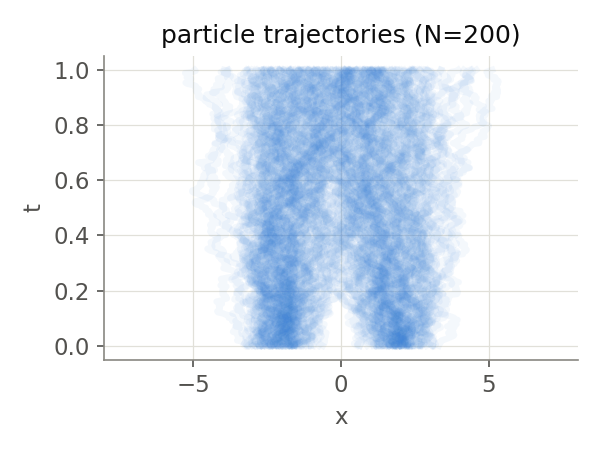} ~\hspace{-.5cm}~ \includegraphics[width=.33\linewidth]{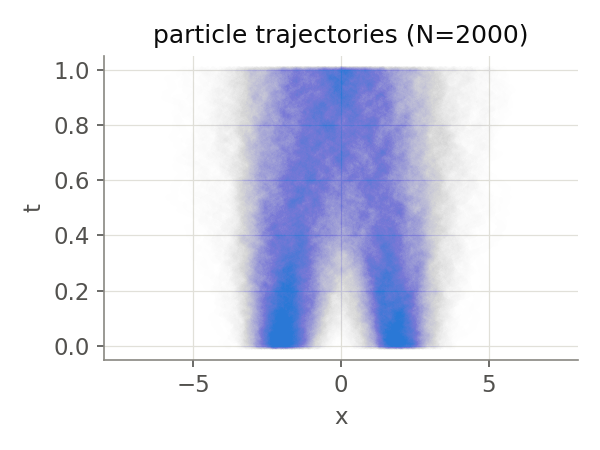}
		\caption{Evolution of the densities in $d=1$: reference density $\rho_t$ (left panel) and  empirical particle densities $\rho_t^N$ for  $N=200$ (middle panel) and $N=2000$ (right panel). Two modes of the initial data are pushed toward the center and diffuse. }
		\label{fig:forward_1D}
	\end{figure}
	\begin{figure}[h]
		\centering\includegraphics[width=.5\linewidth]{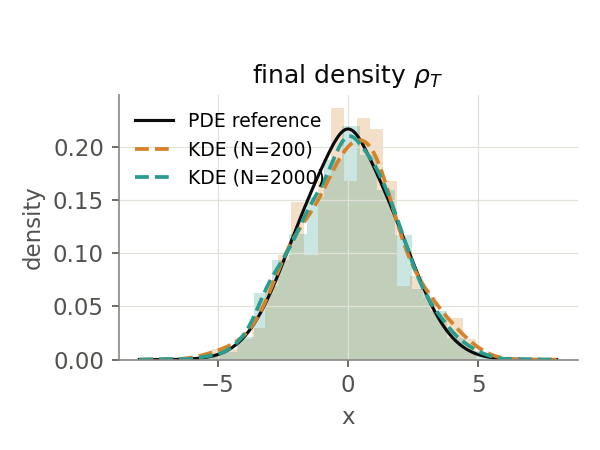}\includegraphics[width=.5\linewidth]{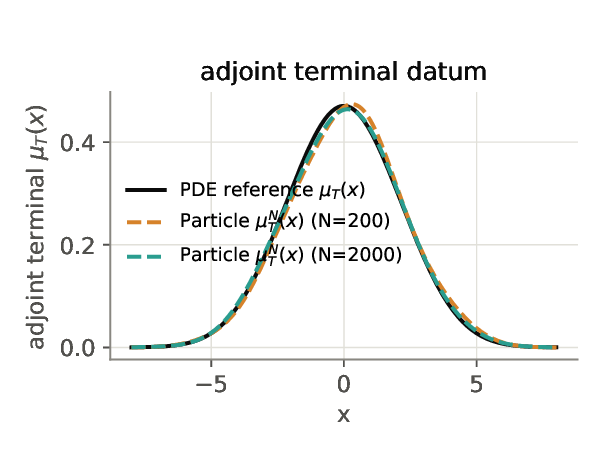}
		\caption{Comparison of final-time  densities in $d=1$: $\rho_T$ (blue) and  $\rho_T^N$, represented as a histogram and a kernel density estimate (orange), for $N\in\{200,2000\}$ particle trajectories (left panel). Corresponding final conditions for the MMD objective \eqref{eq:MMDObjective} (right panel).}
		\label{fig:final_1D}
	\end{figure}
	
	\begin{figure}[h]
		\centering
		\includegraphics[width=1.2\linewidth]{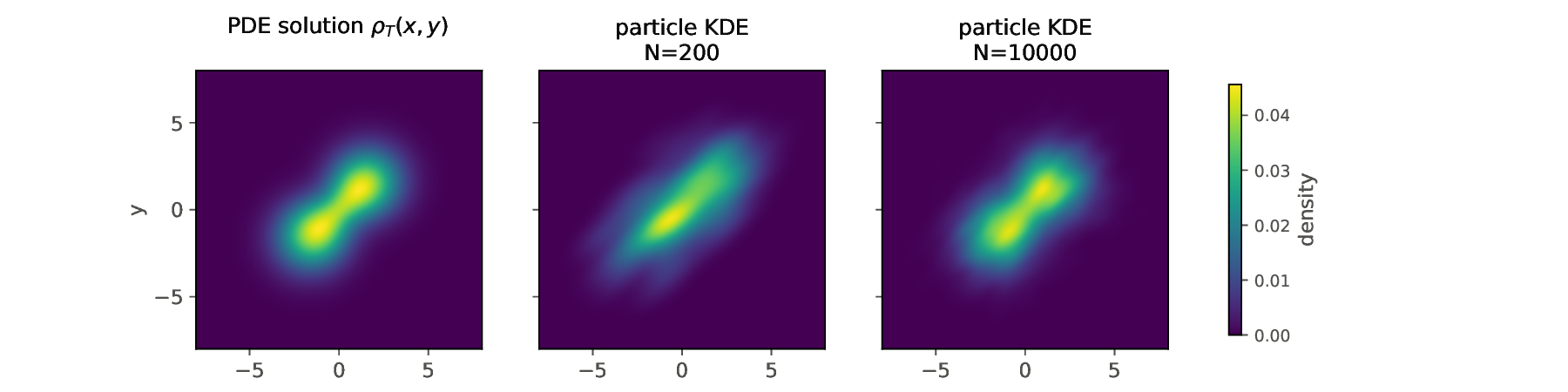}\\ \includegraphics[width=1.2\linewidth]{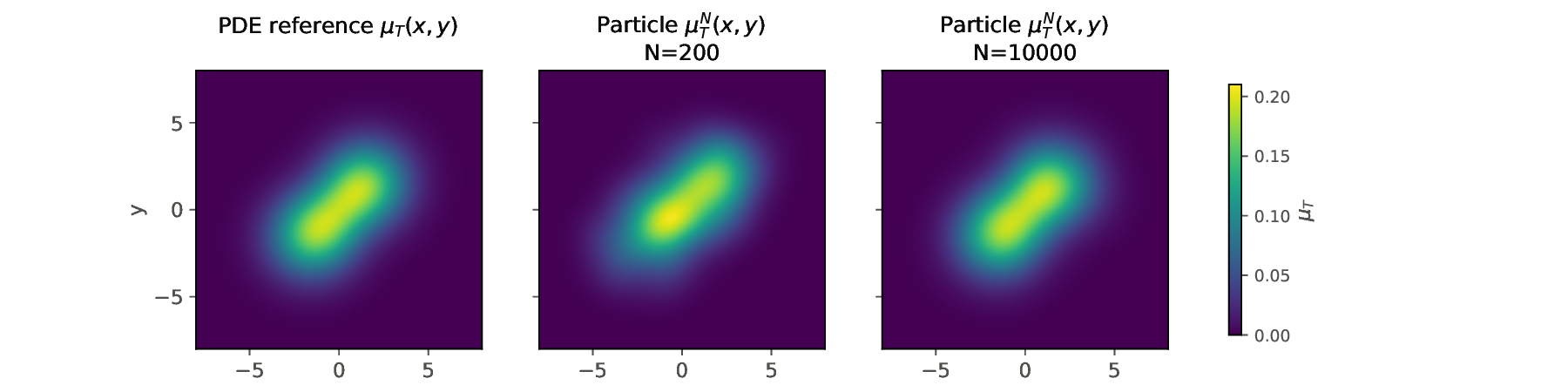}
		\caption{Upper row: Comparison of final-time  densities in $d=2$: PDE reference $\rho_T$ (upper left) and a kernel density estimator of $\rho_T^N$ for $N=200$ (upper middle) and $N=10,000$ (upper right) particle trajectories. The two initial modes are pushed towards the center by the drift term and almost merge.
			\\
			Lower row: Corresponding adjoint final conditions for the MMD objective $(J2)$ based on final states $\mu_T$ (lower left) and $\mu^N_T$ for $N=200$ (lower middle) and $N=10,000$ particles (lower right). Both are Gaussian-mollified versions of the top row final densities. }
		\label{fig:forward_2D}
	\end{figure}

	These  particle approximations lead to  the Eulerian and Lagrangian gradient approximations $\mathcal G^{\mathrm E} [a](x)$ and $\mathcal G^{\mathrm L} [a](x)$ through~\eqref{eq:Eulerian_gradient}  and \eqref{eq:Lagrangial_gradient}, respectively, as displayed in~\Cref{fig:gradient_1D} for dimension $d=1$ and \Cref{fig:gradient_2D} for $d=2$. 
	
	We mention that both the particle and gradient approximations in \Cref{fig:forward_1D,fig:forward_2D} and \Cref{fig:gradient_1D,fig:gradient_2D}  are visibly  coarse for a relatively small number of $N=200$ particles and become more accurate for larger $N$, as expected.  
	
	\begin{figure}[h]
		\centering
		\includegraphics[width=0.8\linewidth]{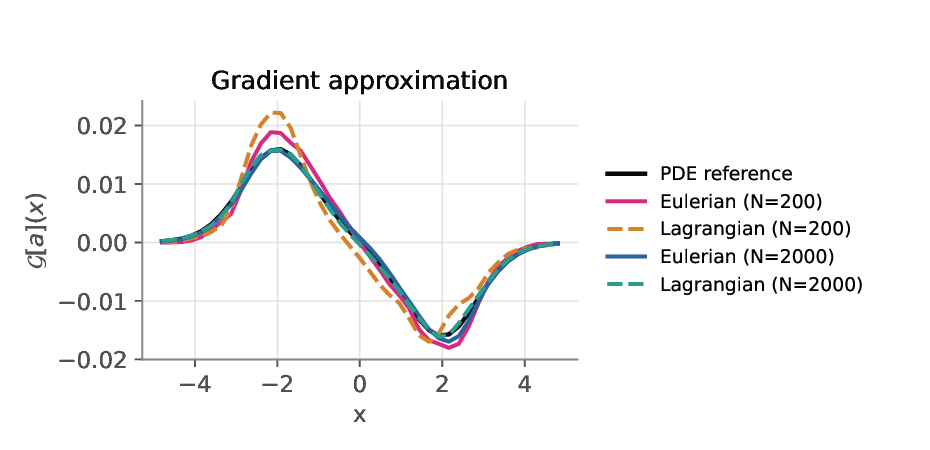}
		\caption{Comparison of the gradient $\mathcal G[a](x)$ and its Eulerian and Lagrangian approximations $\mathcal G^{\mathrm E} [a](x)$ and $\mathcal G^{\mathrm L} [a](x)$ in dimension $d=1$ and under $N\in\{200, 2000\}$ particles. For $2000$, the approximations  almost perfectly overlay  the reference.}
		\label{fig:gradient_1D}
	\end{figure}
	
	\begin{figure}[h]
		
		\hspace{-1cm}\includegraphics[width=1.3\linewidth]{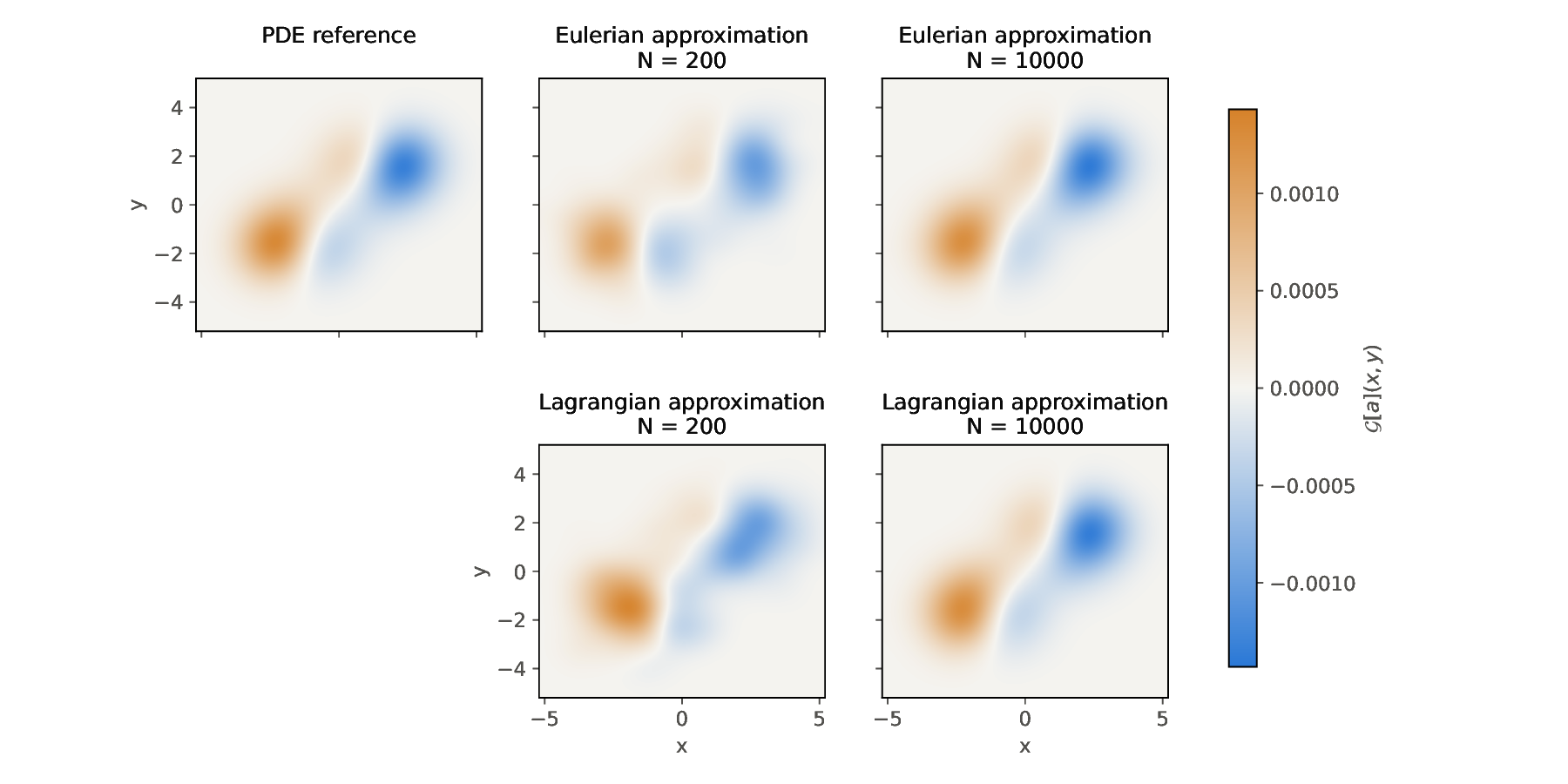}
		\caption{Comparison of the gradient $\mathcal G[a](x,y)$ (upper left) and its Eulerian and Lagrangian approximations $\mathcal G^{\mathrm E}_N [a](x,y)$ (upper row) and $\mathcal G^{\mathrm L}_N [a](x,y)$ (lower row) in dimension $d=2$ and with $N=200$ (center) and $N=10,000$ (right) particles.}
		\label{fig:gradient_2D}
	\end{figure}

	\subsection{Approximation rates}
	In the following, we verify the theoretically predicted convergence rates from \Cref{thm:GmaN}, \Cref{thm:GmiN} and \Cref{prop:error:int,prop:error:EM,prop:error:Lagrange:int,prop:error:Lagrange:EM}. The reference  gradient $\mathcal G[a]$ is computed from  the forward and adjoint Fokker--Planck equations \eqref{eq:FP}--\eqref{eq:adjoint} using the settings described in~\ref{itm:PDEsetting}.
	
	\paragraph{Particle approximation error $\alpha(N)$.}
	To test the particle approximation error, we fix a very small time step $\tau = 2^{-11}$ and sweep $N$ over $2^7,2^8,...,2^{13}$. Errors are evaluated against  $\phi= \phi_2$, averaged over $M=30$ realizations, and reported in \Cref{fig:linear:N_sweep,fig:MMD:N_sweep}. The fitted slopes of regression lines are collected in \Cref{tab:conv-N}. The Lagrangian approximation slopes are close to $-0.5$, even for dimension $d=2$, where our theory predicts the deteriorated convergence rate $\alpha(N) = \mathcal O( N^{-\frac 1 2} \log(1+N))$, which would coincide with a linear slope of $-0.35$ or $-0.32$ in the considered $N$ ranges.  Further investigation would be necessary to distinguish whether this is an artifact of  the noisy nature of the particle methods, or whether the numerical  rate is actually given by the Monte Carlo rate $N^{-1/2}$ across different setups.
	Moreover, we observe  that the Eulerian approximation is more accurate than the Lagrangian approximation in the linear setting in \Cref{fig:linear:N_sweep}, where its adjoint  exactly coincides with the reference adjoint. This advantage is less pronounced for nonlinear objective, as seen in \Cref{fig:MMD:N_sweep}.

	\begin{figure}[h!]
		\centering
		\includegraphics[width=0.5\linewidth]{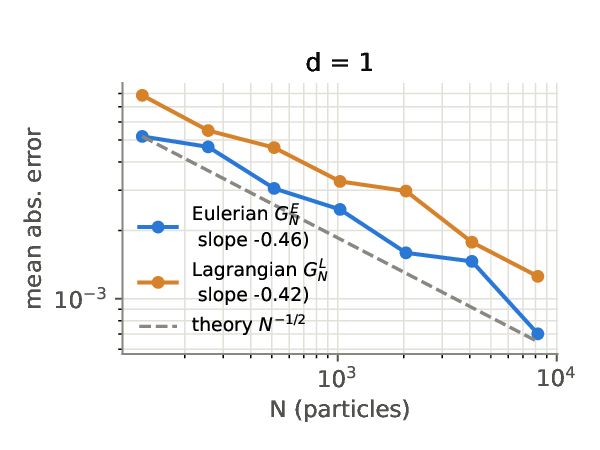}~\includegraphics[width=0.5\textwidth]{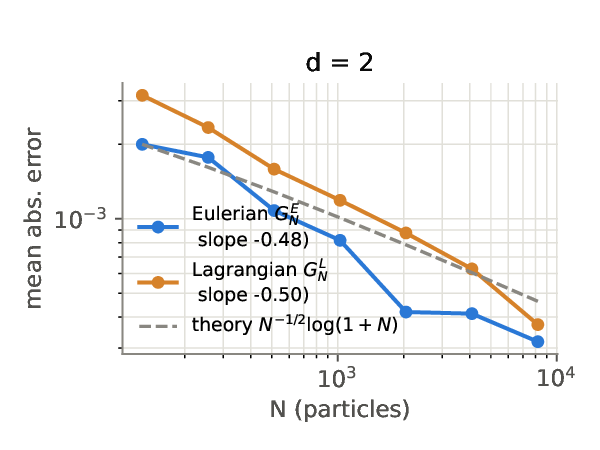}
		\caption{Approximation rates of Eulerian and Lagrangian particle approximations $\mathcal G^{\mathrm E}_N [a]$ and $\mathcal G^{\mathrm L}_N [a]$ in terms of the number of particles $N$, for the linear objective (J1), in dimensions $d=1$ (left) and $d=2$ (right). The dashed line indicates the expected rate according to theory.}
		\label{fig:linear:N_sweep}
	\end{figure}

	\begin{figure}[h!]
		\centering
		\includegraphics[width=0.5\linewidth]{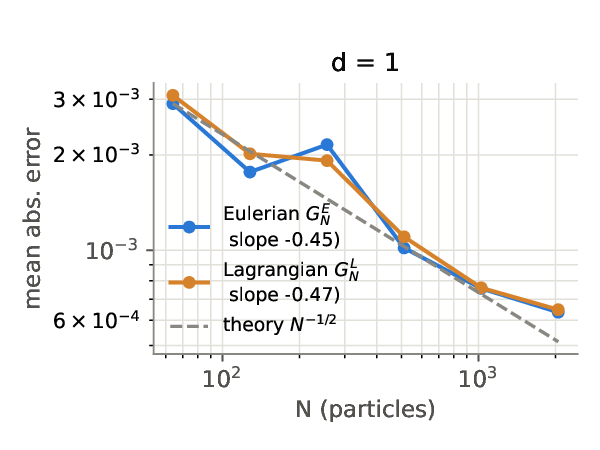}~\includegraphics[width=0.5\textwidth]{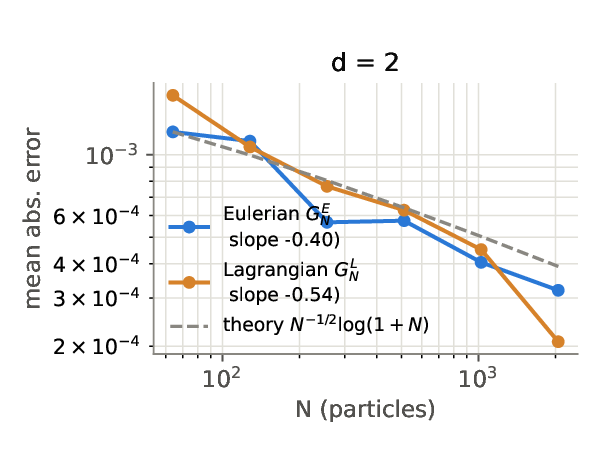}
		\caption{Approximation rates of Eulerian and Lagrangian particle approximations $\mathcal G^{\mathrm E} [a]$ and $\mathcal G^{\mathrm L} [a]$ in terms of the number of particles $N$, for the MMD objective (J2), in dimensions $d=1$ (left) and $d=2$ (right). Mean error over $M=30$ realizations. The dashed line indicates the expected rate according to theory. }
		\label{fig:MMD:N_sweep}
	\end{figure}

	\begin{table}[htbp]
		\centering
		\caption{Empirical $N$-convergence rates at fixed small $\tau = 1/2^{11}$. For $d=2$, the theoretical column displays the
			log-log slope of a line fitted to $\alpha(N)= N^{-\frac 1 2 }\log(1+N)$ over the given $N$ range.}
		\label{tab:conv-N}
		\begin{tabular}{llrrrr}
			\hline
			Objective & $d$ & $N$ range & Theoretical slope & $G_N^E$  & $G_N^L$  \\
			\hline
			Linear (fixed $\psi$)    & 1 & $128$--$8192$ & $-0.500$ & $-0.462$ & $-0.423$ \\
			Linear (fixed $\psi$)    & 2 & $128$--$8192$ & $-0.353$ & $-0.483$ & $-0.496$ \\
			Nonlinear (MMD)          & 1 & $64$--$2048$  & $-0.500$ & $-0.449$ & $-0.465$ \\
			Nonlinear (MMD)          & 2 & $64$--$2048$  & $-0.327$ & $-0.400$ & $-0.542$ \\
			\hline
		\end{tabular}
	\end{table}

	\paragraph{Convergence in high dimensions.} The main motivation for particle methods is their ability to mitigate to the curse of dimensions.  \Cref{fig:LagrangeConv_d10} illustrates convergence of the Lagrangian gradient approximation tested against $\phi_w$ with $w = 2\sqrt{10}$ as the number of particles increases for dimension $d=10$, where a classical grid-based numerical solver for the reference gradient $\mathcal G$ becomes prohibitively expensive.  
	The median stabilizes, while the empirical standard deviation decays approximately with Monte Carlo rate $N^{-1/2}$.

	\begin{figure}
		\centering
		\includegraphics[width=0.5\linewidth]{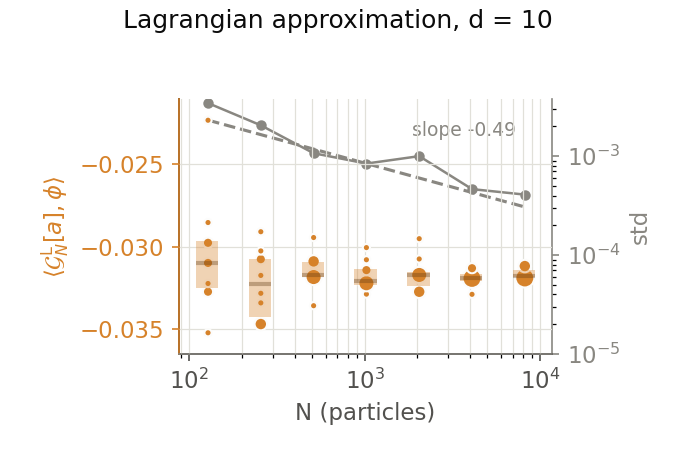}
		\caption{Convergence of $M=10$ realizations of the Lagrangian gradient approximation $\mathcal G^\mathrm{L}_N$ as $N$ grows. Realizations are summarized by dots whose size corresponds to the number of samples it represents. The median is indicated by a brown bar, and the box ranges from the lower to the upper quartile of the samples. The empirical standard deviations of the $M=10$ samples  for each value of  $N$  are plotted in gray, and a fitted dashed line attains slope $-0.49$.
		}
		\label{fig:LagrangeConv_d10}
	\end{figure}
	
	\paragraph{Time discretization error $\tau^{1/2}$.} In order to observe the time discretization error, we use  $N= 2^{15} $ particles. The slopes of fitted regression lines are collected in  \Cref{tab:conv-tau} and all exceed the expected convergence rate $\frac 1 2$.
	
	\begin{table}[htbp]
		\centering
		\caption{Empirical $\tau$-convergence rates at fixed large $N=131072$.}
		\label{tab:conv-tau}
		\begin{tabular}{llrrrr}
			\hline
			Objective & $d$ & $\tau$ range & Theoretical slope & $G_N^E$ & $G_N^L$  \\
			\hline
			Linear (fixed $\psi$) & 1 & $1/16$--$1/512$ & $+0.500$ & $+0.840$ & $+0.687$ \\
			Linear (fixed $\psi$) & 2 & $1/16$--$1/512$ & $+0.500$ & $+0.873$ & $+0.690$ \\
			\hline
		\end{tabular}
	\end{table}
	
	\begin{figure}[h]
		\centering
		\includegraphics[width=0.5\linewidth]{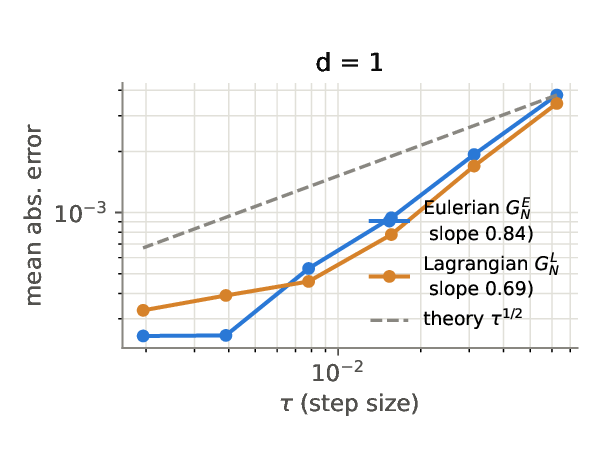}~\includegraphics[width = .5\linewidth]{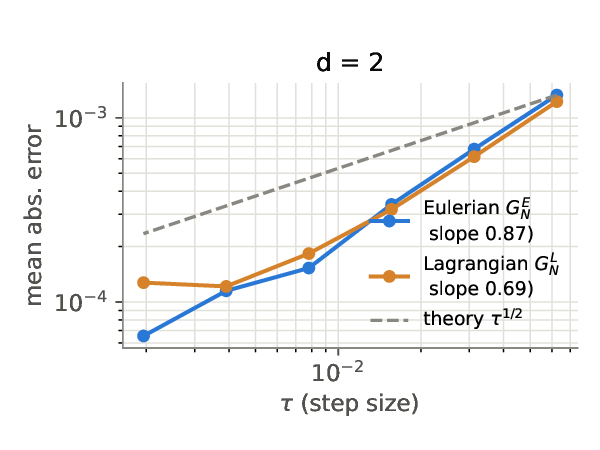}
		\caption{Errors of the Eulerian and Lagrangian numerical particle gradient approximations to $\mathcal G[a]$ for the linear objective (J1) tested against $\phi$, as a function of $\tau$. We choose a large particle number $N=2^{15}$ and errors are averaged over $M=30$ realizations. The dashed line indicates the expected rate according to theory. }
		\label{fig:linear:tau_sweep}
	\end{figure}
	
	\section{Conclusion}
	We investigate two perspectives to construct adjoint solvers for the Fokker--Planck equation, and examine the differences and similarities between Eulerian and Lagrangian formulations. Both formulations provide convergent numerical approximations, and we specify the convergence rates. Our results establish weak convergence, with the error evaluated against a smooth test function. Following the standard Monte Carlo strategy, one should expect a classical $N^{-1/2}$ rate, as also suggested  by the numerical experiments in \Cref{fig:linear:N_sweep,fig:MMD:N_sweep}. Our theory instead supports the dimension-dependent Wasserstein sampling rate $\alpha(N)$, because the gradient is naturally a quadratic quantity and our treatment of the adjoint variable.  We leave the question whether this deteriorated rate is intrinsic to the Monte Carlo approximation of the functional gradient, or merely an artifact of the present analysis, for further investigation.
	Our analysis of the time discretization error generalizes naturally to other choices of numerical discretization schemes for the forward and adjoint particle systems, such as the Milstein scheme. The quadrature error $e_{\mathrm{int}}$, however, emphasizes that temporal regularity of the continuum SDE is a bottleneck that can not easily be overcome. We expect stronger order under weaker notions of convergence. Furthermore, we have so far excluded the PDE discretization error that appears in the computation of $\mu^N$ in the Eulerian gradient approximation from our present analysis. Linearity of the gradient in $\mu$ should allow a direct transfer of this discretization error to the gradient.

	\appendix
		
	\section{A-priori estimates}\label{sec:gman_helper}
		Both the Eulerian and Lagrangian error estimates rely on some a-priori estimates, which we summarize here.
		
		\subsection{Continuum solutions}\label{ssec:apriori:Continuum}
		The forward equation is well-posed under Assumptions~\ref{assumptions:a}, ~\ref{assumptions:rho0}. This follows from standard results \cite{bogachev2022fokker,stroock2007multidimensional} that establish  existence of a probability measure on the path space of the corresponding SDE that translates to a unique measure-valued solution to \eqref{eq:FP}.
		
		\begin{proposition}[Existence of solutions to \eqref{eq:FP}]
			Under Assumptions~\ref{assumptions:a}, ~\ref{assumptions:rho0}, there exists a unique measure-valued solution $\rho\in C([0,T];\mathcal P_q(\mathbb R^d)) $.
		\end{proposition}
		Note that the $q$-th moment bound is inherited from the initial condition by boundedness of the coefficient $a$ and the constant diffusion coefficient. It also implies bounds on all lower moments.
		
		The adjoint equation is also well posed which follows from classical parabolic regularity theory~\cite{oleinik1966mathematical,stroock2007multidimensional, lorenzi2000optimal}.

		\begin{proposition}[Regularity and stability of the adjoint equation]
			\label{prop:adj:Existence+GrdBounds}
			Under Assumptions~\ref{assumptions:a}, ~\ref{assumptions:rho0},~\ref{assumptions:Eulerian}, the backward Kolmogorov equation~\eqref{eq:adjoint}
			admits a unique solution
			\[
			\mu\in C([0,T];C^{2,\theta}_\mathfrak{p}(\mathbb R^d))\cap C^1([0,T];C^{\theta}_\mathfrak{p}(\mathbb R^d)).
			\]
			Moreover, there exists a constant $C$, depending only on $d$, $T$, and the bounds on $a$, such that the following stability estimate holds:
			\[
			\sup_{t\in [0,T]}\|\mu_t\|_{C^{2,\theta}_\mathfrak{p}}
			\le
			C
			\|\mu_T\|_{C^{2,\theta}_\mathfrak{p}},
			\qquad
			0\le t\le T.
			\]
			Furthermore, under additional regularity \ref{assumptions:Euler:additional}, one has 
			\[
			\mu\in C([0,T];C^{3,\theta}_\mathfrak{p}(\mathbb R^d))\cap C^1([0,T];C^{1,\theta}_\mathfrak{p}(\mathbb R^d)).
			\]
		\end{proposition}

		\Cref{thm:GmaN} can be considered as a corollary of~\Cref{prop:grad:continuity} once a Monte Carlo convergence rate is available. This is a consequence of the classical Wasserstein-1 estimate of~\cite{fournier2015rate}, and the fact that $\rho_t = \operatorname{Law}(X_t)$:
		\begin{theorem}\label{thm:samplingRho}
			Let $\rho$ be the solution of \eqref{eq:FP}  and $\rho^N$ the empirical distribution of $N$ i.i.d. particles satisfying \eqref{eq:particle},  under Assumptions~\ref{assumptions:a}, ~\ref{assumptions:rho0}. Then, there exists a constant $C$ that depends on $d$ and the $q$-th moment of $\rho_T$, such that 
			\begin{align*}
				\E[ W_1(\rho_T, \rho^N_T)] \leq& \  C \alpha(N).
			\end{align*}
		\end{theorem}
		
		\subsection{Particle systems}\label{ssec:apriori:particle}
		Standard theory \cite{oksendal2013stochastic, stroock2007multidimensional} 
		ensures existence of the solution to the particle dynamics SDE \eqref{eq:process}, and temporal H\"older regularity of the paths:
		\begin{proposition}\label{prop:Particles:Existence+Bounds}
			
			Under Assumptions~\ref{assumptions:a}, ~\ref{assumptions:rho0}, the SDE \eqref{eq:process} has a unique solution. Its paths are   H\"older continuous: for every $0<p\leq q$, there exists a constant $C_X$ that depends on  $\|a\|_\infty$, $p$, and $T$, such that, for $0\leq r\leq t\leq T$,
			$$\E[\sup_{s\in [r,t]}\left\lvert X_s -X_r\right\rvert^p] \leq C_X(t-r)^{p/2}.$$
		\end{proposition}
		\begin{proof}
			Existence is standard.  H\"older continuity is an immediate consequence of the Burkholder--Gundy--Davis inequality
			\cite{karatzas1998methods}. In particular, from the evolution we obtain 
			\begin{align*}
				\E[\sup_{s\in [r,t]}\left\lvert X_s -X_r\right\rvert^p] & =\E\left[\sup_{s\in [r,t]}\left\lvert \int_r^s a(X_u)\d u + \sqrt{2} (W_u-W_r) \right\rvert ^p\right] \\
				&\leq C_p(\|a\|_\infty^p (t-r)^p + 2^{\frac p 2} C_{BGD}(t-r)^{\frac p 2 }) &\leq C_X (t-r)^{\frac p 2}.
			\end{align*}
		\end{proof}
		
		The classical bound on the  Wasserstein distance  in terms of particle differences 
		\begin{equation}\label{eq:Wasserstein:particles}
			W_1(\rho^N_t, \rho^N_s) = \inf_{\pi \in S^N} \frac 1 N \sum_n\left\lvert X_t^n - X_s^{\pi(n)}\right\rvert\leq  \frac 1 N \sum_n\left\lvert X_t^n - X_s^{n}\right\rvert
		\end{equation}
		then implies H\"older continuity of the empirical particle distribution.
		\begin{corollary}\label{cor:rate:time_particles}
			Under Assumptions~\ref{assumptions:a}, ~\ref{assumptions:rho0}, the empirical distribution function $\rho^N$ of  $N$ of the particles in \eqref{eq:particle} is H\"older continuous with exponent $\frac 1 2 $ w.r.t. the Wasserstein-1 metric:
			\begin{equation}\label{eq:rate:SDEHoelder}
				\E[\sup_{s\in [r,t]}W_1(\rho^N_r, \rho^N_s)]\leq C_X(t-r)^{\frac 1 2}
			\end{equation}
			for the constant $C_X$ from \Cref{prop:Particles:Existence+Bounds} that depends only on $\|a\|_\infty$ and $T$.
		\end{corollary}
		
		As for the adjoint process \eqref{eq:process_adjoint} and particle system \eqref{eq:adjoint_particle}, standard ODE theory guarantees path wise existence of a solution, and its stability w.r.t. the final data. 
		\begin{proposition}\label{prop:adjParticles:Existence+Bounds}
			Let Assumptions~\ref{assumptions:a}, ~\ref{assumptions:rho0},~\ref{assumptions:Lagrangian} hold and consider the adjoint ODE \eqref{eq:process_adjoint}  with fixed $X\in C([0,T]; \mathbb R^d)$ and final condition $\xi\in \mathbb R^d$. It admits a  unique  solution  $Y\in C^1([0,T]; \mathbb R^d)$  and there exists a constant $C_Y'>0$ depending only on $a$ and $T$, such that, 
			for every $0\leq t\leq T$,  
			\begin{equation}\label{eq:Ybound}
				\left\lvert Y_t\right\rvert
				\leq
				C_Y'
				\left\lvert \xi\right\rvert.
			\end{equation}
		\end{proposition}
		By linearity of  \eqref{eq:process_adjoint}, this bound gives stability w.r.t. the final condition.
		\begin{proof}
			Existence on a short time interval follows from Picard-Lindel\"of. 
			For these short times,  the dynamics \eqref{eq:process_adjoint} suggest the implicit bound
			\[
			\left\lvert Y_t\right\rvert
			\leq
			\left\lvert \xi\right\rvert
			+
			\int_t^T
			\|\nabla a(X_s)\|\,\left\lvert Y_s\right\rvert\,\mathrm ds\leq \left\lvert \xi\right\rvert
			+
			\|a\|_{C^1}\int_t^T
			\,\left\lvert Y_s\right\rvert\,\mathrm ds.
			\]
			Gronwall's lemma turns this into the explicit bound \eqref{eq:Ybound}. Global existence  then follows  by  extension on the basis of the non-explosion criterion. 
		\end{proof}
		
		\begin{corollary}\label{cor:adjoint:MomentBd+Lipschitz}
			Under Assumptions~\ref{assumptions:a}, ~\ref{assumptions:rho0},~\ref{assumptions:Lagrangian}
			the adjoint ODE \eqref{eq:adjoint_particle} admits a pathwise unique solution and there exists a $C_Y>0$, independent of $t,s\in [0,T]$ such that
			\begin{enumerate}[(a)] 
				\item $\mathbb E\left\lvert Y_t\right\rvert^2 \leq C_Y$,
				\item $\E \left\lvert Y_t - Y_s\right\rvert\leq C_Y\left\lvert t-s\right\rvert$.
			\end{enumerate}
		\end{corollary}
		
		Note that $C_Y$  depends on  $\|\grd \frac {\delta \mathcal J}{\delta \rho_T}\rvert_{\rho_T}\|_{L^2_{\rho_T}}$ and thus $\rho_T$ through the final condition.
		
		\begin{proof}
			Pathwise existence and stability w.r.t. the final condition follow from the previous lemma. 
			The second moment bound  then directly follows from Assumption~\ref{assumptions:a}, ~\ref{assumptions:rho0},~\ref{assumptions:Lagrangian},  and  implies Lipschitz continuity in time:
			$$\E \left\lvert Y_t - Y_s\right\rvert \leq \E\left\lvert \int_s^t -Y_r \grd a(X_r)\d r \right\rvert  \leq \|\grd a\|_{\infty} \int_s^t \E \left\lvert Y_r\right\rvert\d r. $$
		\end{proof}
		
		\subsection{Time discretization}\label{ssec:apriori:timedisc}
		The convergence rate of the Euler-Maruyama scheme is established in standard texts such as \cite{kloeden1992stochastic}, and we repeat it here for reference.
		\begin{proposition}[\cite{kloeden1992stochastic}]\label{prop:EM:Rate}
			Let $X_t$ be the solution of \eqref{eq:process} and $X_k^\tau$ the solution of the Euler-Maruyama scheme \eqref{eq:EMScheme}. Under Assumptions \ref{assumptions:a},\ref{assumptions:rho0},
			there exists a constant $C_{EM}$ such that 
			\begin{equation}\label{eq:EM:Rate}
				\sup_{k} \mathbb E\left\lvert X_{t_k} - X^\tau_k\right\rvert^2\leq C_{EM}^2\tau.\end{equation}
		\end{proposition}
		By the same particle distance bound as in \eqref{eq:Wasserstein:particles}, one obtains:
		\begin{corollary}\label{corr:EM:RateWasserstein}
			Under assumptions~\ref{assumptions:a},\ref{assumptions:rho0}, the empirical distribution functions $\rho^N$ of the particles \eqref{eq:particle} and their EM discretization \eqref{eq:EMScheme} satisfy
			\begin{equation}\label{eq:EM:RateWasserstein}
				\sup_{k} \mathbb E[W_1(\rho^N_{t_k}, \rho^{N,\tau}_k)]\leq C_{EM}\tau^{1/2}.\end{equation}
		\end{corollary}
		
		\subsection{Proofs of Lemmas~\ref{lem:tildeY} and \ref{lem:EulerScheme:Stab+Acc}}
		We now apply these results to the setting of  $\widetilde{Y}_{i,t}$  introduced in~\eqref{eq:tildeGL} to prove~\Cref{lem:tildeY}.
		\begin{proof}[Proof of~\Cref{lem:tildeY}]
			\begin{enumerate}[(a)]
				\item Follows directly from \Cref{cor:adjoint:MomentBd+Lipschitz}, since the $\widetilde Y_{i,t}$ are i.i.d. copies of $Y_t$.
				\item Exploit stability of the adjoint according to \Cref{prop:adjParticles:Existence+Bounds}
				\begin{align*}
					\mathbb E
					\left\lvert  \widetilde Y_{i,t} - Y_{i,t}
					\right\rvert 
					\leq  & \  C_Y'  \mathbb E
					\left\lvert  \left.\nabla\frac{\delta\mathcal J}{\delta\rho}
					\right\rvert_{\rho_T}
					(X_{i,T}) -\left.\nabla\frac{\delta\mathcal J}{\delta\rho}
					\right\rvert_{\rho_T^N}
					(X_{i,T})\,
					\right\rvert 
					\leq \ C_Y' C_J \mathbb E[W_1(\rho_T, \rho^N_T)] \\
					= & C_Y\alpha(N).
				\end{align*}
				The rate~\eqref{eq:alphaN} follows from~\cite{fournier2015rate} and    $C_J$ is the constant from Assumptions~\ref{assumptions:Lagrangian}.
			\end{enumerate}
			
		\end{proof}

		We similarly prove~\Cref{lem:EulerScheme:Stab+Acc}:
		\begin{proof}[Proof of~\Cref{lem:EulerScheme:Stab+Acc}]\hfill \\
			\begin{enumerate}[(a)]
				\item  Stability follows from  a classical argument using the boundedness of $\grd a$
				\begin{align*}
					\left\lvert\widetilde Y_{i,k}^\tau\right\rvert = & \left\lvert \left(\Pi_{l=k}^{K-1} (1+\tau \grd a(X^\tau_{i,l}) )\right)\widetilde Y_{i,K}^\tau\right\rvert  \\
					\leq & \ (1+\tau \|\grd  a\|_{\infty})^{K-k}\left\lvert\widetilde Y_{i,K}^\tau\right\rvert\leq e^{K\tau \|\grd  a\|_{\infty}}\left\lvert\widetilde Y_{i,K}^\tau\right\rvert&=  e^{T \|\grd  a\|_{\infty}}\left\lvert \zeta \right\rvert.
				\end{align*}
				\item    
				Subtract the Euler scheme \eqref{eq:EulerScheme:Adjoint} from the continuous evolution \eqref{eq:adjoint_particle} to trace the error $e_k:= \widetilde Y_{i,t_k}- \widetilde Y_{i,k}^\tau$:
				\[
				\begin{aligned}
					e_k 
					&= e_{k+1}
					+ \int_{t_k}^{t_{k+1}} \widetilde Y_{i,s}\grd a(X_{i,s}) +(1-1) \widetilde Y_{i,t_{k+1}}\grd a(X_{i,t_k}) 
					\\
					&\hspace{2.2cm}+(1-1) \widetilde Y_{i,t_{k+1}}\grd a(X_{i,k}^\tau)- \widetilde Y_{i,k+1}^\tau\grd a(X_{i,k}^\tau)\d s 
					\\
					&= e_{k+1}(I+\tau \grd a(X_{i,k}^\tau)) + B_1 + B_2.
				\end{aligned}
				\]
				In the following, we bound $\E \left\lvert B_1\right\rvert$ and $\E \left\lvert B_2\right\rvert$ using boundedness and Lipschitzness of $\grd a$.  The  Lipschitzness of $Y$  and H\"older continuity of $X$ in $t$ according to  \Cref{cor:adjoint:MomentBd+Lipschitz,prop:Particles:Existence+Bounds} then show 
				\begin{align*}
					\E \left\lvert B_1\right\rvert= & \ \E\left\lvert \int_{t_k}^{t_{k+1}} \widetilde Y_{i,s} \grd a(X_{i,s}) - \widetilde Y_{i,t_{k+1}} \grd a(X_{i,t_k})\d s\right\rvert  \\
					\leq & \ \int_{t_k}^{t_{k+1}}  \|\grd a\|_{\infty} \E\lvert \widetilde Y_{i,s}  - \widetilde Y_{i,t_{k+1}}\rvert  +  \|\grd^2 a \|_{\infty}\E[\lvert\widetilde Y_{i,t_{k+1}}\rvert \left\lvert X_{i,s}-X_{i,t_{k}}\right\rvert ] \d s \\
					\leq & \|a\|_{C^2} (C_Y\tau^{1/2} + C_Y^{1/2}C_X^{1/2})\tau^{3/2},
				\end{align*}
				
				Furthermore, strong convergence of the Euler--Maruyama scheme \cite{kloeden1992stochastic}, as repeated in \Cref{prop:EM:Rate}, 
				implies
				\begin{align*}
					\E \left\lvert B_2\right\rvert= & \ \E\left\lvert \int_{t_k}^{t_{k+1}} \widetilde  Y_{i,t_{k+1}} (\grd a(X_{i,t_k})- \grd a(X_{i,k}^\tau))\d s\right\rvert  \\ 
					\leq &\ \|\grd^2 a\|_\infty \int_{t_k}^{t_{k+1}} \E[\lvert \widetilde Y_{i,t_{k+1}}\rvert\lvert X_{i,t_k}-X_{i,k}^\tau\rvert]\d s \quad  \leq \quad  \| a\|_{C^2} C_Y^{1/2} C_{EM}\tau^{3/2}.
				\end{align*}
				Consequently, after collecting all  constants into $C$, one obtains
				\[
				\mathbb E\left\lvert e_k\right\rvert
				\leq
				(1+C\tau)\mathbb E\left\lvert e_{k+1}\right\rvert
				+
				C\tau^{3/2}
				\]
				The  recursion resolves to
				\begin{align*}
					\mathbb E\left\lvert e_k\right\rvert \leq & \ (1+C\tau)\mathbb E\left\lvert e_K\right\rvert +C\tau^{3/2} \sum_{i= 0} ^{K-k-1} (1+C\tau)^i =\  C\tau^{3/2} \sum_{i= 0}^{K-k-1} (1+C\tau)^i\\
					\leq & \ C\tau^{3/2} K (1+C\tau)^K  \quad \leq C T e^{CT} \tau^{1/2} \quad =:\quad  C_{\widetilde Y^\tau}'\tau^{1/2},
				\end{align*}
				uniformly in $k$, because $\mathbb E\left\lvert e_K\right\rvert= 0$ and $(1+C\tau)^K\leq e^{CK\tau}$ with $K\tau = T$.
			\end{enumerate}
		\end{proof}

\section*{Funding}
KH acknowledges support from the Jet Propulsion Laboratory PDRDF 24AW0133. QL acknowledges support from ONR-N000142612095, NSF-DMS-2608503 and Vilas associate award. YY is partially supported by National Science Foundation under Grant DMS-2409855 and DMS-2540324, and Office of Naval Research Award N00014-24-1-2088. 

\section*{Acknowledgements}
		The authors would like to thank Ludovico Theo Giorgini for unveiling the connection to the fluctuation-dissipation theorem.
	
\section*{Data Availability}
	
	Data will be made available on reasonable request.

	\bibliographystyle{abbrv}
	\bibliography{lit.bib}
	
\end{document}